%% file: sample_paper.tex
\documentclass[twoside]{article}
\usepackage{amsmath, amssymb}
\usepackage[most]{tcolorbox}
\tcbset{
  colback=gray!5,
  colframe=gray!60!black,
  boxrule=0.5pt,
  arc=2pt,
  left=4pt,
  right=4pt,
  top=2pt,
  bottom=2pt,
  fonttitle=\bfseries,
  coltitle=black,
}

\newtcolorbox{propbox}[1][]{
  title={#1},
  enhanced,
  sharp corners,
}
\usepackage[preprint]{aistats2026}
\usepackage{graphicx}
\usepackage{subcaption} 
\usepackage{placeins}  

\usepackage{amsmath,amssymb,amsthm,mathtools}
\usepackage{fullpage}
\usepackage{xcolor}   
\usepackage{hyperref}
\usepackage{amsmath}
\usepackage{bm}       
\usepackage{bbm}       
\theoremstyle{plain}
\newtheorem{lemma}{Lemma}[section]
\newtheorem{theorem}[lemma]{Theorem}

\usepackage{algorithm}

\usepackage{algpseudocode}  

\theoremstyle{definition}

\theoremstyle{remark}

\newtheorem{assumption}{Assumption}

\input{math_commands}

\usepackage{cleveref}
\begin{document}

%

%

\twocolumn[

\aistatstitle{Stochastic Optimization with Random Search}

\aistatsauthor{ El Mahdi Chayti$^{(1)}$ \And Taha El Bakkali El Kadi$^{(2)}$ \And Omar Saadi$^{(2)}$ \And  Martin Jaggi$^{(1)}$ }

\aistatsaddress{(1) Machine Learning and Optimization Laboratory (MLO), EPFL \And $\qquad\qquad$ (2) UM6P College of Computing } ]

\begin{abstract}
We revisit random search for stochastic optimization, where only noisy function evaluations are available. We show that the method works under weaker smoothness assumptions than previously considered, and that stronger assumptions enable improved guarantees. In the finite-sum setting, we design a variance-reduced variant that leverages multiple samples to accelerate convergence. Our analysis relies on a simple translation invariance property, which provides a principled way to balance noise and reduce variance.
\end{abstract}

\section{Introduction}
Zeroth-order (derivative-free) optimization refers to a class of algorithms that minimize an objective function using only function evaluations. These methods are crucial in modern machine learning applications, where gradients are either inaccessible, unreliable, or too costly to obtain. Such scenarios arise in black-box adversarial attacks~\cite{chen2017zoo, ughi2022empirical}, hyperparameter tuning~\cite{koch2018autotune, turner2021bayesian}, reinforcement learning with policy gradient estimation~\cite{malik2020derivative, mania2018simple}, and simulation-based optimization where only function evaluations are accessible~\cite{conn2009introduction}.

Within this setting, two distinct algorithmic paradigms have emerged. The first, often called gradient-approximation methods, constructs surrogate gradient estimates using finite differences and then mimics first-order updates~\cite{nesterov2017random, ghadimi2013stochastic, liu2018zeroth}. Although these approaches inherit convergence guarantees from gradient-based optimization, they introduce discretization bias and require careful tuning of smoothing parameters, which can significantly affect performance in practice.

In contrast, the second paradigm, known as direct-search, operates purely through comparisons of function values at a set of trial points, and updates to the best point without ever forming a gradient estimate~\cite{bergou2020stochastic, golovin2020gradientless}. This approach is simple to implement and often more robust to noise in function evaluations.

The \emph{Stochastic Three Points} (STP) method of ~\cite{bergou2020stochastic} 
is one of the simplest and most efficient forms of direct-search methods, requiring only three function evaluations 
per iteration and incorporating random search directions (this is where random search comes from). It evaluates the objective at the current iterate and at two origin-symmetric perturbations along a random direction, then selects the best of the three. In the deterministic setting, STP achieves optimal complexity bounds: $O(d/\varepsilon^2)$ in the smooth non-convex case~\cite{bergou2020stochastic}, matching the best-known rate achieved by gradient estimation methods, and $O(d/\varepsilon)$ in the smooth convex case~\cite{kadi2025on}, matching the best known complexity bound among gradient estimation methods that do not use momentum. 

Despite their appeal, STP and random search methods have been analyzed almost exclusively in the deterministic setting, where access to the full objective function is assumed. However, this assumption is often unrealistic in large-scale machine learning, where the objective typically involves a finite sum over millions of data points. Large-scale learning tasks such as deep reinforcement learning or empirical risk minimization require stochastic algorithms. While stochastic gradient estimation methods are well studied, stochastic random search methods remain poorly understood. Yet, their simplicity and freedom from discretization bias make them an attractive alternative to gradient-based approaches. This gap in the theory motivates the development of stochastic variants of random search that can operate effectively with noisy or subsampled objective estimates. Providing sharp guarantees for such methods is therefore both theoretically and practically important. To address this limitation, in \cite{boucherouite2024mistp}, they proposed \emph{Minibatch STP} (MiSTP), which replaces the true objective with a minibatch estimate during the three-point comparison of standard STP. Their analysis, however, yields a suboptimal $O(d^3/\varepsilon^6)$ complexity under strong individual smoothness assumptions, significantly worse than the $O(d/\varepsilon^4)$ rate achieved by gradient-estimation methods like RSGF \cite{ghadimi2013stochastic}. The goal of this work is to close this gap.
\section{Related Work}

\textbf{Direct search and STP.}  
Classical direct-search methods (see \cite{conn2009introduction,vicente2013worst}) rely on deterministic search directions and tend to scale poorly with the dimensionality. In contrast, the STP algorithm achieves improved complexity bounds: $\mathcal{O}(d/\epsilon^2)$ in the smooth nonconvex setting~\cite{bergou2020stochastic}, and $\mathcal{O}(d/\epsilon)$ in the smooth convex setting~\cite{kadi2025on}. These results improve upon the $\mathcal{O}(d^2/\epsilon^2)$ and $\mathcal{O}(d^2/\epsilon)$ complexities obtained for deterministic direct-search methods~\cite{vicente2013worst,konevcny2014simple}. Follow-up work extended STP with importance sampling~\cite{bibi2020stochastic} and momentum~\cite{gorbunov2020stochastic}, but these analyses were conducted only for deterministic objective functions.

\textbf{STP in the stochastic setting.}  
To the best of our knowledge, the only prior attempt to analyze STP for stochastic objectives is due to Boucherouite et al.~\cite{boucherouite2024mistp}. Their minibatch STP method achieves a complexity bound of order $\mathcal{O}(d^3/\varepsilon^6)$ under strong \emph{individual smoothness} assumptions, requiring each stochastic component function $f_\xi$ to be smooth. This rate is substantially worse than the $\mathcal{O}(d/\varepsilon^4)$ complexity known for gradient-estimation-based methods such as RSGF~\cite{ghadimi2013stochastic}, which approximate gradients using finite differences.

\textbf{Gradient-estimation methods.}  
Zeroth-order stochastic methods based on gradient approximation (e.g., two-point estimators 
\cite{ghadimi2013stochastic}, ZO-SVRG \cite{liu2018zeroth}) achieve a complexity bound of 
$\cO(d/\varepsilon^4)$ under individual smoothness assumption.  However, these methods introduce discretization bias and require careful tuning of smoothing parameters, which can significantly impact their empirical performance. In contrast, random search methods avoid discretization errors altogether but have historically lacked comparable theoretical guarantees in the stochastic setting.

\textbf{Our contribution in context.}  
Our contributions can be summarized as follows:  
\begin{itemize}\setlength{\itemsep}{2pt}\setlength{\parskip}{0pt}\setlength{\parsep}{0pt}
    \item \textbf{Individual smoothness.}  
    We prove that Random Search achieves the optimal
    $\cO(d/\varepsilon^4)$ complexity under the standard
    \emph{individual smoothness} assumption, while retaining its bias-free and
    comparison-based update rule.
    \item \textbf{Average smoothness.}  
    Under the weaker assumption that only the average function $f$ is
    $(L_0,L_1)$-smooth, we show that Random Search recovers the
    $\cO(d^3/\varepsilon^6)$ rate previously obtained
    in~\cite{boucherouite2024mistp}, but under strictly milder requirements.  
    \item \textbf{Finite-sum variance reduction.}  
    For finite-sum objectives, we demonstrate that variance reduction can be
    realized without the memory overhead of storing snapshots as in
    ZO-SVRG~\cite{liu2018zeroth}. We establish a complexity of
    $\cO\!\left(\min\!\Big\{\tfrac{d^{4/3} n^{2/3}}{\varepsilon^{8/3}}, \tfrac{dn}{\varepsilon^2}\Big\}\right)$,
    which improves upon the deterministic rate whenever $n \gg d/\varepsilon^2$.
    \item \textbf{Helper or human feedback.}  
    We extend Random Search to settings where only \emph{inexact comparison
    feedback} is available, such as A/B testing~\cite{kohavi2009controlled} or
    RLHF~\cite{christiano2017deep,ouyang2022training}. We prove that convergence
    holds up to an accuracy floor of order $\cO(\sqrt{d \delta})$, where $\delta$
    quantifies the inexactness of the helper or human feedback.
\end{itemize}

Together, these results provide the first unified analysis of Random Search
across stochastic, finite-sum, and human-in-the-loop settings.

\section{General Algorithmic Framework}
We consider the stochastic minimization problem
\[
\min_{x \in \mathbb{R}^d} f(\vx), \qquad 
f(\vx) := \mathbb{E}_{\xi \sim \mathcal{P}}[f_{\xi}(\vx)],
\]
where $(\Xi,\mathcal{F},\mathcal{P})$ is a probability space and the mapping $F$ defined on $\mathbb{R}^d \times \Xi$ by $F(x,\xi)=f_{\xi}(\vx)$ satisfies:
(i)~$x \mapsto F(x,\xi)$ is differentiable for every $\xi \in \Xi$;
(ii)~$\xi \mapsto F(x,\xi)$ is $\mathcal{F}$-measurable for every $x \in \mathbb{R}^d$.
We assume $f$ is finite-valued and bounded below.  
This framework includes the finite-sum setting common in machine learning: for $\Xi = \{1,\dots,n\}$, uniform $\mathcal{P}$, and $F(x,i) = f_i(\vx)$, we recover $f(\vx) = \frac{1}{n}\sum_{i=1}^n f_i(\vx)$. $\|\cdot\|$ stands for the $\ell_2$ norm $\|\cdot\|_2$.
\subsection{Smoothness assumption}

We adopt the following notion of smoothness, introduced in~\cite{zhang2020gradientclippingacceleratestraining}.

\begin{assumption}[$(L_0,L_1)$-smoothness]
\label{ass:smoothness}
A differentiable function $f:\mathbb{R}^d \to \mathbb{R}$ is $(L_0,L_1)$-smooth if, $\forall \vx,\vy \in \mathbb{R}^d:$
\begin{multline}
    \|\nabla f(\vx) - \nabla f(\vy)\|
    \;\leq\; \big(L_0 + L_1 \|\nabla f(\vx)\|\big)\, \|\vx-\vy\|.
\end{multline}

\end{assumption}

This condition generalizes standard Lipschitz smoothness by allowing the gradient 
Lipschitz constant to grow proportionally with the gradient norm itself. It has been 
used to justify adaptive methods such as gradient clipping~\cite{zhang2020gradientclippingacceleratestraining}.

We note that if a function \( f : \mathbb{R}^d \to \mathbb{R} \) is \((L_0, L_1)\)-smooth (i.e., satisfies Assumption~\ref{ass:smoothness}), then for all \(\vx, \vy \in \mathbb{R}^d\), we have:
\begin{equation}\label{ineq:L0,L1}
\begin{aligned}
f(\vy) &\leq f(\vx) + \langle \nabla f(\vx), \vy - \vx \rangle 
+ \frac{L_0 + L_1 \|\nabla f(\vx)\|}{2} \\
&\quad \times \|\vx - \vy\|^2.
\end{aligned}
\end{equation}
\subsection{Random directions}

At each iteration the algorithm samples a random direction $\vs_t \sim \mathcal{D}$. 
We require the following assumption, adapted from 
\cite{bergou2020stochastic, boucherouite2024mistp}.

\begin{assumption}[Exploration property of $\mathcal{D}$]
\label{ass:distribution}
The distribution $\mathcal{D}$ on $\mathbb{R}^d$ satisfies:
\begin{enumerate}
    \item Normalization: $\E_{\vs \sim \mathcal{D}}[\|\vs\|^2] = 1$.
    \item Exploration: there exists $\mu_{\mathcal{D}} > 0$ and a norm $\|\cdot\|$ such that for all $\vg \in \mathbb{R}^d$,
    \begin{equation}
        \E_{\vs \sim \mathcal{D}} \big[\, |\langle \vg, \vs \rangle| \,\big] 
        \;\geq\; \mu_{\mathcal{D}} \|\vg\|.
    \end{equation}
\end{enumerate}
\end{assumption}

This ensures that the distribution of search directions has sufficient coverage 
in expectation. In prior work, the norm $\|\cdot\|$ was allowed to depend on $\mathcal{D}$; 
here, by equivalence of norms in finite dimensions, we absorb this dependence 
into the constant $\mu_{\mathcal{D}}$.

\subsection{Algorithm}
\textbf{Stochastic Three Points (STP).}
STP is a classical derivative-free randomized search method introduced in \cite{bergou2020stochastic}.
At iteration $t$, given an iterate $\vx_t\in\mathbb{R}^d$, a stepsize $\alpha_t>0$, and a search direction distribution $\mathcal{D}$ on $\mathbb{R}^d$, the method draws $\vs_t\sim\mathcal{D}$ and forms the two symmetric trial points
$
\vx_t^{+} = \vx_t + \alpha_t \vs_t, \,\,
\vx_t^{-} = \vx_t - \alpha_t \vs_t,
$
then evaluates the objective at the three points $\{\vx_t^{-},\,\vx_t,\,\vx_t^{+}\}$, and updates according to the best-of-three rule:
\[
\vx_{t+1} \in \arg\min_{\vx \in \{\vx_t^{-},\,\vx_t,\,\vx_t^{+}\}} f(\vx).
\]

\textbf{Minibatch STP (MiSTP).}
In the finite-sum setting, MiSTP \cite{boucherouite2024mistp} adapts STP by replacing the full objective $f$ with a single empirical objective $f_{\mathcal{C}_t}$, computed from a minibatch $\mathcal{C}_t$, while retaining the three-point comparison:
$
\vx_{t+1} \in \arg\min_{\vx \in \{\vx_t^{-},\,\vx_t,\,\vx_t^{+}\}} f_{\mathcal{C}_t}(\vx).$
Because the accept/reject decision is made using $f_{\mathcal C_t}$, the selected step may worsen the true objective $f$, and thus including the current point $\vx_t$ does not ensure anymore the monotonic improvement of the true objective as it is the case for classical STP. We drop this baseline evaluation at $\vx_t$ and adopt a two–point comparison: 
\[
\vx_{t+1} \in \arg\min_{\vx \in \{\vx_t^{-},\,\vx_t^{+}\}} f_{\mathcal C_t}(\vx),
\]
which evaluates both trial points on the same minibatch and avoids the noisy baseline at $\vx_t$. By denoting $M_t^+$ and $M_t^-$ the stochastic estimators of $f(\vx_t^+)$ and $f(\vx_t^-)$, we introduce the following algorithm: 

\begin{algorithm}[ht]
\caption{Stochastic Random Search (general form)}
\label{alg:srs}
\begin{algorithmic}[1]
\Require initial point $\vx_0 \in \mathbb{R}^d$, 
         step sizes $\{\eta_t\}_{t \geq 0}$, distribution $\mathcal{D}$
\For{$t = 0,1,2,\dots$}
    \State Sample direction $\vs_t \sim \mathcal{D}$
    \State Compute $\vx_t^+ = \vx_t + \eta_t \vs_t$ and 
           $\vx_t^- = \vx_t - \eta_t \vs_t$
    \State Obtain stochastic estimates $M_t^+$ and $M_t^-$ 
           of $f(\vx_t^+)$ and $f(\vx_t^-)$
    \State Update
    \[
        \vx_{t+1} = \vx_t - \eta_t \, 
        \mathrm{sign}(M_t^+ - M_t^-) \, \vs_t
    \]
\EndFor
\end{algorithmic}
\end{algorithm}

This is a natural stochastic extension of the classical STP update: the method 
moves along $\vs_t$ in the direction suggested by the lower of the two estimated 
function values.

In general, the quantities $M_t^\pm$ may be constructed in any way that provides
useful estimates of $f(\vx_t^\pm)$; this includes subsampling-based evaluations
(Section~\ref{sec:stochastic-subsampling}), variance-reduced finite-sum
constructions (Section~\ref{sec:finite-sum}), and human or helper feedback
(Section~\ref{sec:helper}). More sophisticated mechanisms, such as
momentum-based surrogates, also fit naturally into this framework.

\subsection{General descent bound}

The following lemma shows that, under Assumption~\ref{ass:smoothness}, Algorithm~\ref{alg:srs} 
guarantees a descent step up to the stochastic approximation error.

\begin{lemma}[Descent bound]
\label{lem:descent}
Suppose $f$ is $(L_0,L_1)$-smooth and Assumption~\ref{ass:distribution} holds. 
If $\eta_t \leq \mu_{\mathcal{D}}/L_1$, then
\begin{align}
    f(\vx_{t+1}) \;\leq\; &f(\vx_t) 
    - \mu_{\mathcal{D}} \eta_t \|\nabla f(\vx_t)\|
    + \tfrac{L_0}{2} \eta_t^2 \notag \\
    &+ 2\,\E\!\left[\, 
        |M_t^+ - f(\vx_t^+)| 
        + |M_t^- - f(\vx_t^-)| 
    \,\right].
\end{align}
\end{lemma}

The first three terms correspond to the smoothness-based descent typical in random search, 
while the last term captures the additional error introduced by the stochastic 
estimators $M_t^\pm$.

\subsection{Translation invariance}\label{translation_subsec}

A crucial property of Algorithm~\ref{alg:srs} is \emph{translation invariance}: 
adding the same constant $c_t$ to both $M_t^+$ and $M_t^-$ does not change 
the outcome. This allows us to optimally shift the estimates and reduce the error term. 
Formally,
\begin{align}
\inf_{c \in \mathbb{R}} 
&\Big( |M_t^+ - c - f(\vx_t^+)| + |M_t^- - c - f(\vx_t^-)| \Big) \notag \\
&= \tfrac{1}{2} \, \big| (M_t^+ - M_t^-) - ( f(\vx_t^+) - f(\vx_t^-)) \big|.\label{eq5:TransSym}
\end{align}

Thus, rather than controlling each approximation error separately, it suffices 
to control the error in the \emph{difference} $M_t^+ - M_t^-$, which is typically 
much smaller. This observation will be the key to our improved rates.
\section{General Stochastic Functions with Subsampling} \label{sec:stochastic-subsampling}

We consider the general stochastic optimization problem
\[
    f(\vx) \;=\; \E_{\xi}\big[f_\xi(\vx)\big],
\]
where only noisy evaluations of $f_\xi$ are available. To reduce noise, we
estimate function values at the perturbed points
$\vx_t^\pm = \vx_t \pm \eta_t \vs_t$ using minibatches of size $b$:
\[
   M_t^\pm := \frac{1}{b} \sum_{j=1}^b f_{\xi_j}(\vx_t^\pm),
   \qquad \{\xi_j\}_{j=1}^b \ \text{i.i.d.}
\]

We analyze two regimes of smoothness assumptions.

\subsection{Case I: Weak average smoothness}

We first assume that only the mean function $f$ satisfies
$(L_0,L_1)$-smoothness (Assumption~\ref{ass:smoothness}). To control the
stochastic error, we assume a bounded variance of the component functions
(cf.~\cite{boucherouite2024mistp}):
\begin{assumption}[Bounded variance]
\label{ass:variance}
There exists $\sigma_0^2<\infty$ such that for all $\vx\in\R^d$, $\E[f_\xi(\vx)] = f(\vx)$ and
\(
   \E[\left(f_\xi(\vx) - f(\vx)\right)^2] \le \sigma_0^2.
\)
\end{assumption}

Consequently, for $M = \frac{1}{b}\sum_{i=1}^b f_{\xi_i}(\vx)$, the average of $b$ i.i.d realizations of $f_{\xi}(\vx)$, we have:  $\E|M - f(\vx)| \le \sigma_0/\sqrt{b}$.

Under Assumptions~\ref{ass:smoothness}, \ref{ass:distribution},
and~\ref{ass:variance}, Lemma~\ref{lem:descent} gives, for
$\eta_t\le \mu_{\mathcal{D}}/L_1$,
\begin{align}
    f(\vx_{t+1}) \;\le\;& f(\vx_t) 
    - \mu_{\mathcal{D}} \eta_t \|\nabla f(\vx_t)\|
    + \tfrac{L_0}{2}\eta_t^2 \notag \\
    &+ 2\,\E\!\left[\, 
        |M_t^+ - f(\vx_t^+)| 
        + |M_t^- - f(\vx_t^-)| 
    \,\right].
\end{align}
By Assumption~\ref{ass:variance}, $\E|M_t^\pm - f(\vx_t^\pm)| \le
\sigma_0/\sqrt{b}$. Summing over $t$ and averaging gives:

\begin{theorem}[Average smoothness]
\label{thm:avg-caseI}
Let $F_0:=f(\vx_0)-f^\star<\infty$. With $\eta_t\equiv \eta \le
\mu_{\mathcal{D}}/L_1$,
\[
    \frac{1}{T}\sum_{t=0}^{T-1}\E\|\nabla f(\vx_t)\|
    \;=\;\frac{1}{\mu_{\mathcal{D}}}
    \cO\!\Big(
        \frac{F_0}{\eta T}
        + L_0\,\eta
        + \tfrac{\sigma_0}{\eta\sqrt{b}}
    \Big).
\]
Optimizing $\eta$ yields
\[
    \frac{1}{T}\sum_{t=0}^{T-1}\E\|\nabla f(\vx_t)\|
    \;=\;\frac{1}{\mu_{\mathcal{D}}}
    \cO\!\Big(
        \frac{L_1 F_0}{\mu_{\mathcal{D}} T}
        + \sqrt{\tfrac{L_0 F_0}{T}}
        + \tfrac{\sqrt{L_0\sigma_0}}{b^{1/4}}
    \Big).
\]
\end{theorem}

To ensure $\varepsilon$-accuracy, that is, $\frac{1}{T}\sum_{t=0}^{T-1}\E\|\nabla f(\vx_t)\| \leq \varepsilon$, it suffices to pick
\[
    T = \O\!\Big(\tfrac{dL_1}{\varepsilon}
            + \tfrac{dL_0F_0}{\varepsilon^2}\Big),
    \qquad
    b = \O\!\Big(\tfrac{(dL_0\sigma_0)^2}{\varepsilon^4}\Big),
\]
using $\mu_{\mathcal{D}}\asymp 1/\sqrt{d}$. The total number of calls is
\[
    \textsc{Calls} \;=\; 2bT
    = \O\!\Bigg(
        \Big(\frac{dL_1}{\varepsilon}
            + \frac{dL_0F_0}{\varepsilon^2}\Big)
        \cdot
        \frac{(dL_0\sigma_0)^2}{\varepsilon^4}
    \Bigg).
\]
The dominant term is therefore
\(\tilde{\cO}(d^3\sigma_0^2/\varepsilon^6)\).

\subsection{Case II: Standard sample smoothness}

We now assume that each component function $f_\xi$ is $(L_0,L_1)$-smooth with
respect to $\nabla f$:
\begin{equation}
\label{eq:sample-smoothness-wrt-f}
    \|\nabla f_\xi(\vx) - \nabla f_\xi(\vy)\|
    \le (L_0+L_1 \|\nabla f(\vx)\|)\,\|\vx-\vy\|.
\end{equation}
This is the \emph{standard assumption} in the analysis of stochastic
zeroth-order methods (see, e.g.,~\cite{ghadimi2013stochastic}, which corresponds
to the special case $L_1=0$). Our formulation is slightly more general.

We also assume bounded gradient variance:
\begin{assumption}[Bounded gradient variance]
\label{ass:grad-var}
There exists $\sigma_1^2<\infty$ such that for all $\vx\in\R^d$, $\E\nabla f_\xi(\vx) = \nabla f(\vx)$, and 
\(
    \E\|\nabla f_\xi(\vx) - \nabla f(\vx)\|^2 \le \sigma_1^2.
\)
\end{assumption}

Consequently, we show that for $M = \frac{1}{b}\sum_{i=1}^b \nabla f_{\xi_i}(\vx)$ is the average of an i.i.d minibatch of size $b$ independent of $\vs|\vx$, and $\vs\sim\mathcal{D}$ isotropic (like the uniform distribution over the unit sphere, or normal distribution) we have
\(
    \E\big[\,|\langle M - f(\vx),\vs\rangle|\,\big] =\cO\left( \sigma_1 / \sqrt{d b}\right).
\)

Combining~\eqref{eq:sample-smoothness-wrt-f} with
Assumption~\ref{ass:grad-var} and translation symmetry gives:

\begin{lemma}[Difference approximation error]
\label{lem:diff-error}
With $\vx_t^\pm=\vx_t\pm\eta\vs_t$ and $\E\|\vs_t\|^2=1$,
\begin{multline*}
    \E\big|(M_t^+-M_t^-) - (f(\vx_t^+)-f(\vx_t^-))\big|
   \\\;\le\; 2(L_0+L_1\|\nabla f(\vx_t)\|)\eta^2
   + \cO\!\left(\tfrac{\eta\sigma_1}{\sqrt{db}}\right).
\end{multline*}
\end{lemma}

This yields the following bound:

\begin{theorem}[Sample smoothness]
\label{thm:avg-caseII}
Let $F_0:=f(\vx_0)-f^\star<\infty$. Under
Assumptions~\ref{ass:distribution}, \ref{ass:grad-var}, and sample smoothness,
for $\eta\le \mu_{\mathcal{D}}/(5L_1)$ we have
\[
   \tfrac{1}{T}\sum_{t=0}^{T-1}\E\|\nabla f(\vx_t)\|
   = \tfrac{1}{\mu_{\mathcal{D}}}
   \cO\!\left(\tfrac{F_0}{\eta T}+L_0\eta+\tfrac{\sigma_1}{\sqrt{db}}\right).
\]
Optimizing $\eta$ yields
\[
   \tfrac{1}{T}\sum_{t=0}^{T-1}\E\|\nabla f(\vx_t)\|
   = \tfrac{1}{\mu_{\mathcal{D}}}
   \cO\!\left(\tfrac{L_1F_0}{\mu_{\mathcal{D}}T} +
   \sqrt{\tfrac{L_0F_0}{T}} + \tfrac{\sigma_1}{\sqrt{db}}\right).
\]
\end{theorem}

To guarantee $\varepsilon$-accuracy, again, this means $\frac{1}{T}\sum_{t=0}^{T-1}\E\|\nabla f(\vx_t)\| \leq \varepsilon$, it suffices to pick
\[
   T = \O\!\Big(\tfrac{dL_1}{\varepsilon} + \tfrac{dL_0F_0}{\varepsilon^2}\Big),
   \qquad
   b = \O\!\Big(\tfrac{\sigma_1^2}{\varepsilon^2}\Big).
\]
The total complexity is
\[
    \textsc{Calls} \;=\; 2bT
    = \O\!\Bigg(
        \Big(\frac{dL_1}{\varepsilon}
            + \frac{dL_0F_0}{\varepsilon^2}\Big)
        \cdot
        \frac{\sigma_1^2}{\varepsilon^2}
    \Bigg).
\]
The dominant term is therefore
\(\tilde{\cO}(d\sigma_1^2/\varepsilon^4)\).

\subsection{Significance of the results}

We now compare the two regimes and highlight the improvements our analysis provides.

\paragraph{Weak average smoothness.}
In this regime we match the complexity
$\tilde{\cO}(d^3\sigma_0^2/\varepsilon^6)$ of~\cite{boucherouite2024mistp}, but
under a much weaker assumption: only the \emph{average} function $f$ is required
to be $(L_0,L_1)$-smooth, whereas they assumed each $f_\xi$ is globally
$L$-smooth. Moreover, our analysis shows that the dominant term depends only on
$L_0$, which can be much smaller than $L=L_0+GL_1$ where $G$ bounds
$\|\nabla f\|$.

\paragraph{Standard sample smoothness.}
Under the standard assumption that each $f_\xi$ is $(L_0,L_1)$-smooth with
respect to $\nabla f$ (as in~\cite{ghadimi2013stochastic}, $L_1=0$), we obtain
the improved complexity $\tilde{\cO}(d\sigma_1^2/\varepsilon^4)$. Thus compared
to the weak case, both the dependence on dimension and on accuracy improve. This
shows that random search achieves the same scaling as gradient estimation
methods, while avoiding discretization bias and remaining conceptually simpler.

\medskip
In the next section, we show how variance reduction in the finite-sum setting
can further improve the complexity when the dataset size $n$ is large.

\section{Finite-Sum Case and Variance Reduction}\label{sec:finite-sum}

We now consider the finite-sum setting
\[
    f(\vx) \;=\; \frac{1}{n}\sum_{i=1}^n f_i(\vx),
\]
Now we ask of each $f_i$ to be G-Lipschitz. In this regime, we can apply variance
reduction by periodically evaluating the full dataset.

\paragraph{Variance-reduced scheme.}
Let $m$ be a fixed epoch length. Every $m$ iterations we perform a full pass
over the dataset to compute exact values of the objective at the perturbed
points. In between, we use stochastic minibatches of size $b$.  
Formally, the estimators are defined as
\[
   M_t^\pm =
   \begin{cases}
      f(\vx_t^\pm), & \text{if } t \equiv 0 \pmod{m}, \\[6pt]
      \tfrac{1}{b}\sum_{j=1}^b f_{\xi_j}(\vx_t^\pm),
      & \text{otherwise},
   \end{cases}
\]
where $\{\xi_j\}_{j=1}^b$ are sampled i.i.d.\ from $\{1,\dots,n\}$.

\paragraph{Translation-symmetric version.}
Ordinarily, variance reduction methods require storing a \emph{snapshot}
$\tilde{\vx}$ and using it to form a control variate correction for subsequent
iterations. In our case, translation symmetry makes this correction automatic,
since only the difference $M_t^+ - M_t^-$ matters. Thus we do not need to store
$\tilde{\vx}$ explicitly.  

\paragraph{Two-snapshot version.}
For clarity, one can equivalently consider maintaining two separate snapshots
$\tilde{\vx}^+$ and $\tilde{\vx}^-$, corresponding to the last time
$\vx_t^\pm$ was evaluated on the full dataset:
\[
   M_t^\pm =
   \begin{cases}
      f(\vx_t^\pm),  \qquad\text{if } t \equiv 0 \pmod{m}, \quad \text{else } \\[6pt]
      f(\tilde{\vx_t}^\pm) 
      + \tfrac{1}{b}\sum_{j=1}^b \Big( f_{\xi_j}(\vx_t^\pm) - f_{\xi_j}(\tilde{\vx_t}^\pm) \Big).
   \end{cases}
\]
This is the classical variance-reduced form, where the stochastic correction
uses the difference between current and snapshot minibatches. It yields the same
theoretical guarantees and total complexity as the translation-symmetric
formulation, but requires storing two snapshots in memory. In contrast, the
translation-symmetric construction avoids this storage burden.

\paragraph{Error bound.}
As in the sample smoothness regime, we analyze the difference
$M_t^+ - M_t^-$. Using translation symmetry and the variance-reduction scheme,
one can show
\begin{align}
    \E\Big[
    \big|&(M_t^+ - M_t^-)
    - (f(\vx_t^+) - f(\vx_t^-))\big|
    \Big]\nonumber\\
    \;&\leq\;
     \cO\!\left(\frac{\eta\,m\,G}{\sqrt{ b}}\right),
\end{align}
where $\tilde{\vx}_t$ denotes the most recent snapshot at which the function
was fully evaluated.

\paragraph{Averaged gradient bound.}
Plugging this refined error into Lemma~\ref{lem:descent} and choosing
$\eta\leq \mu_{\mathcal{D}}/(5L_1)$, we obtain
\begin{align}
    \frac{1}{T}\sum_{t=0}^{T-1}\E\|\nabla f(\vx_t)\|
    \;=\; \frac{1}{\mu_{\mathcal{D}}}\,
    \O\!\Big(
        \frac{F_0}{\eta T}
        + L_0\eta
        + \frac{G m}{\sqrt{ b}}
    \Big).
\end{align}
Optimizing over $\eta$ yields
\begin{align}
    \frac{1}{T}\sum_{t=0}^{T-1}\E\|\nabla f(\vx_t)\|
    \;=\; \frac{1}{\mu_{\mathcal{D}}}\,
    \O\!\Big(
        \frac{L_1 F_0}{\mu_{\mathcal{D}} T}
        + \sqrt{\frac{L_0 F_0}{T}}
        + \frac{G m}{\sqrt{ b}}
    \Big).
\end{align}
As before, $\mu_{\mathcal{D}}\asymp 1/\sqrt{d}$ for isotropic distributions.

\paragraph{Parameter choices.}
To ensure accuracy $\varepsilon$, it suffices to pick
\begin{align}
    T &= \O\!\Big(\tfrac{dL_1}{\varepsilon}
        + \tfrac{dL_0F_0}{\varepsilon^2}\Big),
    &
    b(m) &= \O\!\Big(\tfrac{ d m^2 G^2}{\varepsilon^2}\Big).
\end{align}

\paragraph{Total complexity.}
Each block of $m$ iterations requires
\[
    \text{one full pass: } n
    \quad+\quad
    (m-1)\cdot b(m) \ \text{stochastic calls}.
\]
Thus, over $T$ iterations the total number of function evaluations is
\[
    \textsc{Calls}(m)
    \;=\; \frac{T}{m}\Big(n + (m-1)b(m)\Big).
\]

Optimizing over $m$ gives the following dominant complexity term:
\begin{equation}
    \textsc{Calls} \;=\;
    \cO\!\Big(
        \frac{L_0F_0\,d^{4/3}\,n^{2/3}\,G^{2/3}}
             {\varepsilon^{8/3}}
    \Big).
\end{equation}

\paragraph{Discussion.}
This rate is strictly better than the deterministic complexity
$\cO(dnL_0F_0/\varepsilon^2)$ whenever
\(
    n \;\geq\; d G^2/\varepsilon^2.
\)
Both the single-snapshot (translation-symmetric) and the two-snapshot variants
achieve the same total complexity, but the former avoids the need to store
snapshots in memory, which can be critical in large-scale settings. Variance
reduction, therefore, provides a practical and theoretical advantage when the
dataset size $n$ is large relative to the noise level, bringing the complexity
closer to the optimal regime of stochastic optimization.

\section{Learning with Helper or Human Feedback}\label{sec:helper}

In many applications, direct function evaluations may not be available or
desirable, but one can obtain indirect feedback in the form of preferences or
comparisons. Examples include:
\begin{itemize}
    \item \textbf{A/B testing:} where we cannot measure the exact loss function,
    but can observe user preferences between two versions of a product or
    service~\cite{kohavi2009controlled}.
    \item \textbf{Reinforcement learning with human feedback (RLHF):} where the
    reward signal is provided via human preference comparisons rather than
    explicit numerical scores~\cite{christiano2017deep, ouyang2022training}.
\end{itemize}
In such scenarios, our framework extends naturally by incorporating a
\emph{helper} (e.g., a human labeler, a proxy model, or a heuristic) that
produces feedback about the relative quality of two states.

\paragraph{$\delta$-inexact feedback.}
We formalize the quality of the helper by requiring that for all $\vx,\vy$,
\[
    \E\big[\,\big|h(\vx) - h(\vy) - (f(\vx) - f(\vy))\big|\,\big]
    \;\leq\; \delta,
\]
that is, the helper provides a $\delta$-accurate approximation of the true
difference in function values. We call this property
\emph{$\delta$-inexact feedback}. This is similar in spirit to the helper framework introduced in \cite{chayti2024optimization} for first-order methods and \cite{chayti2024unified} for second-order methods.

\paragraph{Theoretical guarantee.}
Using $\delta$-inexact feedback in place of exact evaluations in the
stochastic three-point method, we can establish the following guarantee.

\begin{theorem}[Convergence with $\delta$-inexact feedback]
\label{thm:helper}
Let $F_0=f(\vx_0)-f^\star<\infty$ and assume
$(L_0,L_1)$-smoothness (Assumption~\ref{ass:smoothness}) and
directional distributional exploration
(Assumption~\ref{ass:distribution}). With constant step size
$\eta_t\equiv \eta \le \mu_{\mathcal{D}}/L_1$, we have
\[
    \frac{1}{T}\sum_{t=0}^{T-1}\E\|\nabla f(\vx_t)\|
    \;=\;\frac{1}{\mu_{\mathcal{D}}}
    \cO\!\Big(
        \frac{F_0}{\eta T}
        + L_0\,\eta
        + \frac{\delta}{\eta}
    \Big).
\]
Optimizing over $\eta$ yields
\[
    \frac{1}{T}\sum_{t=0}^{T-1}\E\|\nabla f(\vx_t)\|
    \;=\;\frac{1}{\mu_{\mathcal{D}}}
    \cO\!\Big(
        \frac{L_1 F_0}{\mu_{\mathcal{D}} T}
        + \sqrt{\tfrac{L_0 F_0}{T}}
        + \sqrt{L_0\delta}
    \Big).
\]
\end{theorem}

\paragraph{Accuracy guarantee.}
By choosing
\[
    T \;=\; \O\!\Big(\tfrac{dL_1}{\varepsilon} + \tfrac{dL_0F_0}{\varepsilon^2}\Big),
\]
we ensure that
\[
    \frac{1}{T}\sum_{t=0}^{T-1}\E\|\nabla f(\vx_t)\|
    \;\leq\; \varepsilon \;+\; \cO(\sqrt{d\delta}).
\]
In other words, convergence can only be guaranteed up to a neighborhood of size
$\cO(\sqrt{d\delta})$, which reflects the intrinsic inexactness of the helper’s
feedback.

\paragraph{Discussion.}
This result shows that inexact comparison feedback degrades convergence only
through the additive term $\sqrt{d\delta}$. When $\delta$ is small, the
algorithm remains effective and achieves rates comparable to those with exact
function evaluations. In practice, this provides theoretical support for
random-search-based optimization in settings such as RLHF and A/B testing,
where exact evaluations are unavailable but pairwise feedback is abundant.

\section{On the Limitations of Momentum for Random Search}

A natural question is whether the variance-reduction ideas developed for
first-order stochastic optimization can be transferred to Random Search. In the
first-order literature, several approaches exploit momentum to mitigate
variance, including Heavy Ball~\cite{polyak1964some}, momentum-based variance
reduction such as STORM~\cite{cutkosky2019momentum}, and implicit gradient
transport~\cite{arnold2019reducing}. These methods help precisely when gradient
noise is high and batches are small.

\paragraph{Naive adaptation.}
In Random Search we never access (nor approximate) gradients; the algorithm
operates purely via function evaluations and their comparisons. It is tempting
to adapt first-order momentum by reusing the same update forms but replacing
stochastic gradients with function \emph{differences}. For example, the Heavy
Ball moving average
\(
  M_t = (1-\beta)M_{t-1} + \beta\,\nabla f_\xi(\vx_t)
\)
translates to
\[
  M_t \;=\; (1-\beta)M_{t-1} \;+\; \beta\big(f_\xi(\vx_t^+) - f_\xi(\vx_t^-)\big),
\]
and more elaborate momentum-based variance-reduction or implicit-transport
schemes can be adapted analogously by plugging function differences in place of
stochastic gradients.

\paragraph{Negative result (why the classical analysis fails).}
In first-order methods, momentum introduces a bias term, \emph{but} the
resulting variance reduction typically dominates, yielding net gains. In our
comparison-based setting, that trade-off collapses: the stochastic term in the
difference estimator scales as
\(
\mathrm{Std}\big[f_\xi(\vx_t^+)-f_\xi(\vx_t^-)\big] \,\propto\, \eta,
\)
because $\vx_t^+$ and $\vx_t^-$ are only $\Theta(\eta)$ apart. At the same
time, the \emph{signal} itself,
\(
f(\vx_t^+)-f(\vx_t^-) \approx 2\eta\,\langle\nabla f(\vx_t),\vs_t\rangle,
\)
also scales linearly with $\eta$. Consequently, the signal-to-noise ratio of
the per-iteration difference is essentially \emph{scale-invariant} in $\eta$,
and exponential averaging (momentum) does not buy the same variance advantage
as in the gradient case. Worse, the momentum recursion accumulates a
\emph{smoothing bias} that depends on how the differences
$\,f(\vx_t^+)-f(\vx_t^-)\,$ drift across iterations; with random directions and
non-stationary iterates, this bias does not shrink with $\eta$ at a rate that
beats the $\Theta(\eta)$ stochastic term. The standard first-order proofs that
balance bias and variance, therefore break: reducing the variance via momentum
no longer outpaces the induced bias, unless one resorts to \emph{large
minibatches}, which negates the intended benefit.

\paragraph{Takeaway.}
While momentum is a central tool for first-order methods, its direct
transcription to Random Search—simply replacing gradients with a function
differences—does not yield theoretical or practical improvements. The core
obstruction is the step-size–dependent structure of both signal and noise in
the difference oracle. Any meaningful momentum-like gain in this
comparison-based regime will likely require \emph{new} constructions (beyond
literal translations of first-order schemata) that explicitly account for the
$\Theta(\eta)$ scaling of both the signal and the noise.

\section{Experiments}\label{sec:experiments}

\paragraph{Setup.}
We compare \emph{Mi2P}, \emph{RSGF}, and \emph{ZO\text{-}CD} on a logistic-loss
classification task using the Breast Cancer Wisconsin dataset. After a
train/test split, the training set has \(n=455\) samples and \(d=30\) features.
For $y_i\in\{\pm 1\}$, we optimize
\[
f(\vx)=\tfrac{1}{n}\sum_{i=1}^n \log\!\bigl(1+\exp(-y_i\langle a_i,x\rangle)\bigr)
+ \tfrac{\lambda}{2n}\|x\|^2.
\]
For batch sizes \(b\in\{1,5,10,25,50,100\}\) we run 20 trials each. Step sizes
are tuned once per \(b\) in a pilot phase, then fixed. All methods are compared
under the same \emph{function-query budget}, ensuring aligned $x$-axes in plots.
We report \(f(\vx)\) versus queries (mean $\pm$ one s.d.\ over runs).

\paragraph{Methods.}
Each iteration evaluates a minibatch objective \(F_B\) with $|B|=b$:
\begin{itemize}\setlength{\itemsep}{0pt}
\item \textbf{RSGF:} sample $u$ on the unit sphere, set
\(g=\tfrac{F_B(x+\mu u)-F_B(\vx)}{\mu}\,u\), update $x\gets x-\eta g$.
\item \textbf{ZO\text{-}CD:} coordinate-wise two-point estimator
\(g=\sum_{i=1}^d \tfrac{F_B(x+\mu e_i)-F_B(x-\mu e_i)}{2\mu}\,e_i\),
then $x\gets x-\alpha g$.
\end{itemize}
Mi2P uses only the sign of $F_B(x{+}\alpha s)-F_B(x{-}\alpha s)$.
Per iteration, Mi2P and RSGF cost $2b$ queries, ZO\text{-}CD costs $2bd$.

\paragraph{Results.}
Figure~\ref{fig:bc_all} shows $f(\vx)$ versus queries. For very small batches
($b=1,5$), Mi2P underperforms due to noisy comparisons that can flip the
ordering of $x{+}\alpha s$ and $x{-}\alpha s$. As $b$ increases, comparisons
become reliable: Mi2P matches or surpasses RSGF and consistently outperforms
ZO\text{-}CD. This aligns with the theory that the misranking probability
decreases with $b$, yielding more accurate updates.

\begin{figure}[H]
  \centering
  \begin{subfigure}[t]{0.32\linewidth}
    \includegraphics[width=\linewidth]{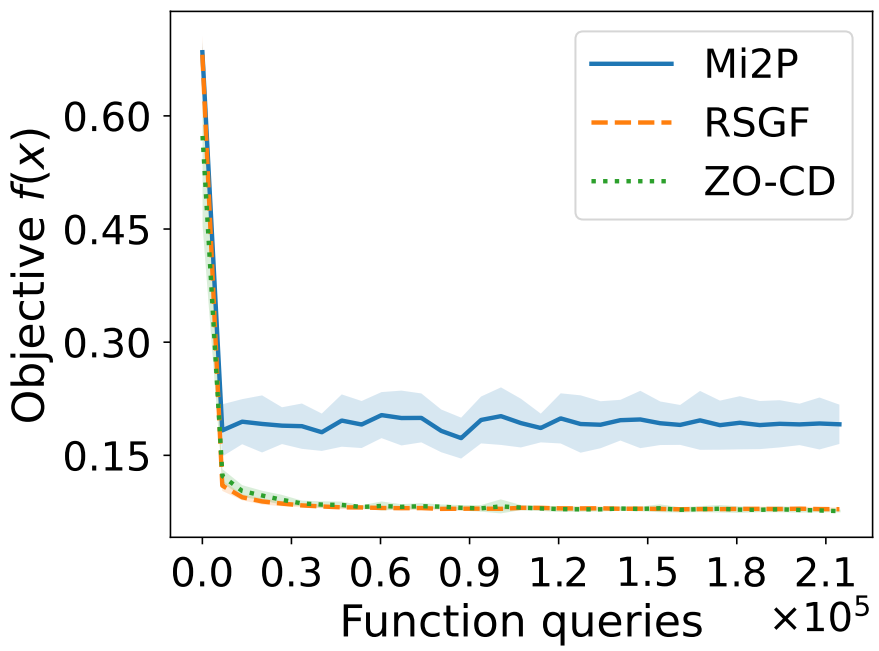}
    \caption{$b=1$}
  \end{subfigure}\hfill
  \begin{subfigure}[t]{0.32\linewidth}
    \includegraphics[width=\linewidth]{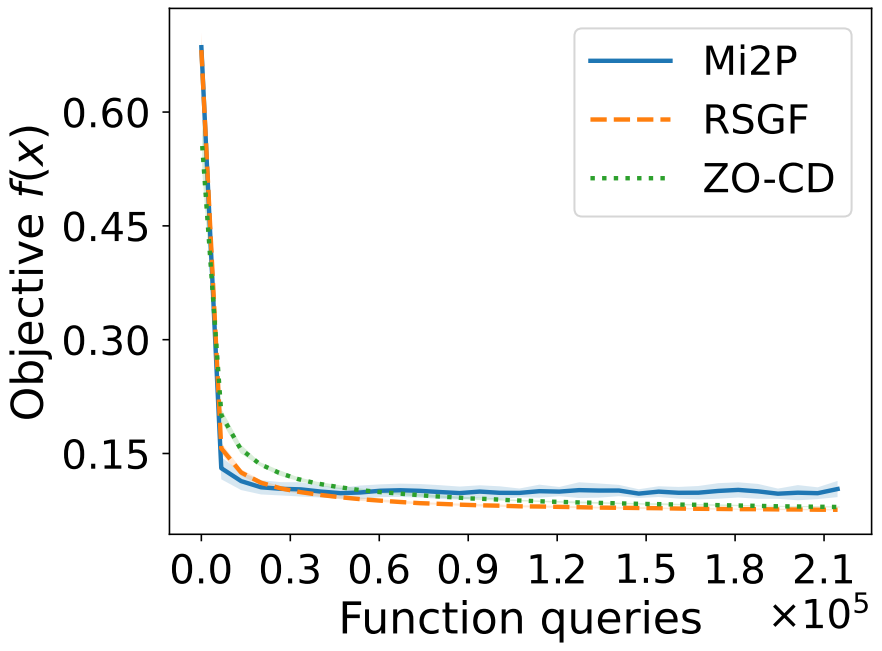}
    \caption{$b=5$}
  \end{subfigure}\hfill
  \begin{subfigure}[t]{0.32\linewidth}
    \includegraphics[width=\linewidth]{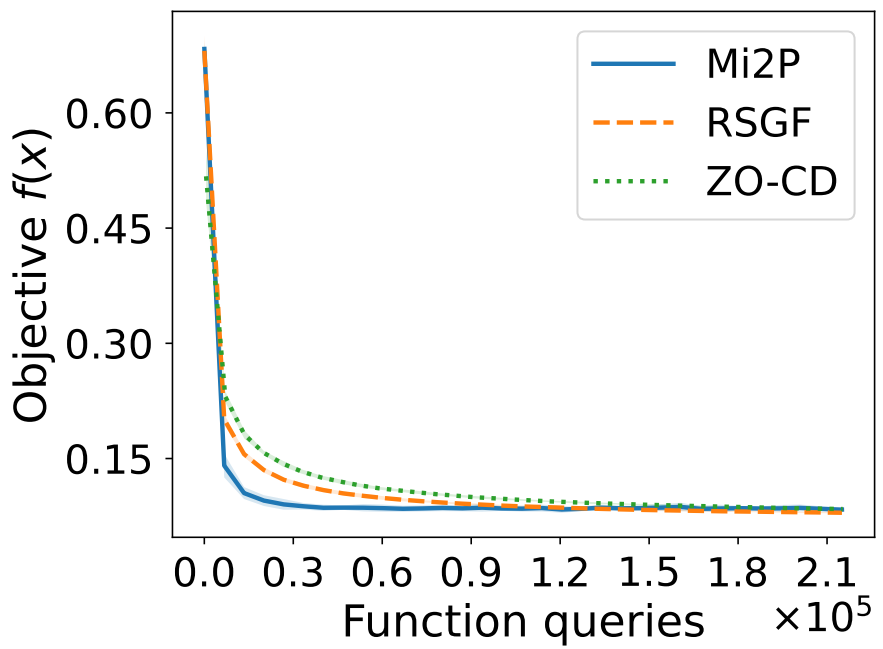}
    \caption{$b=10$}
  \end{subfigure}\\[4pt]
  \begin{subfigure}[t]{0.32\linewidth}
    \includegraphics[width=\linewidth]{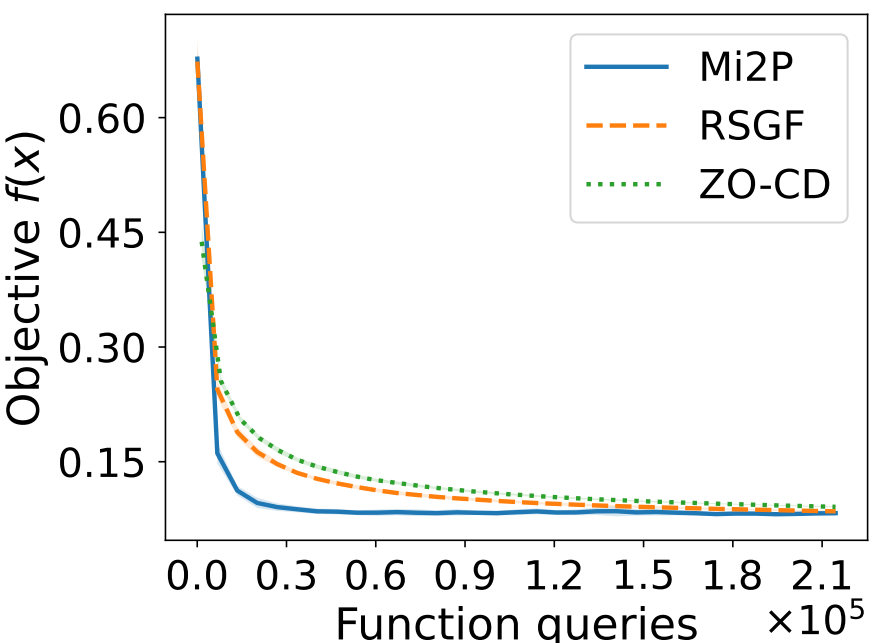}
    \caption{$b=25$}
  \end{subfigure}\hfill
  \begin{subfigure}[t]{0.32\linewidth}
    \includegraphics[width=\linewidth]{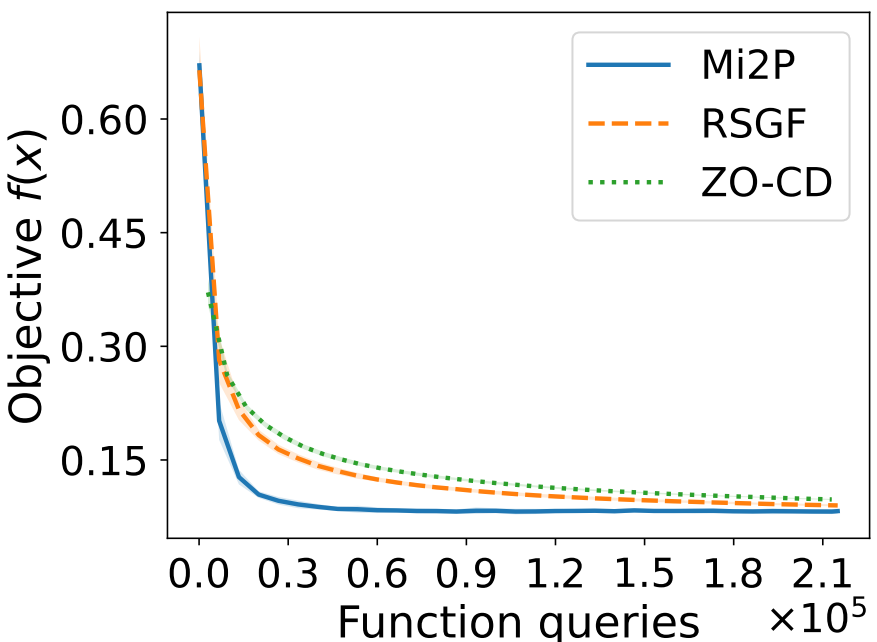}
    \caption{$b=50$}
  \end{subfigure}\hfill
  \begin{subfigure}[t]{0.32\linewidth}
    \includegraphics[width=\linewidth]{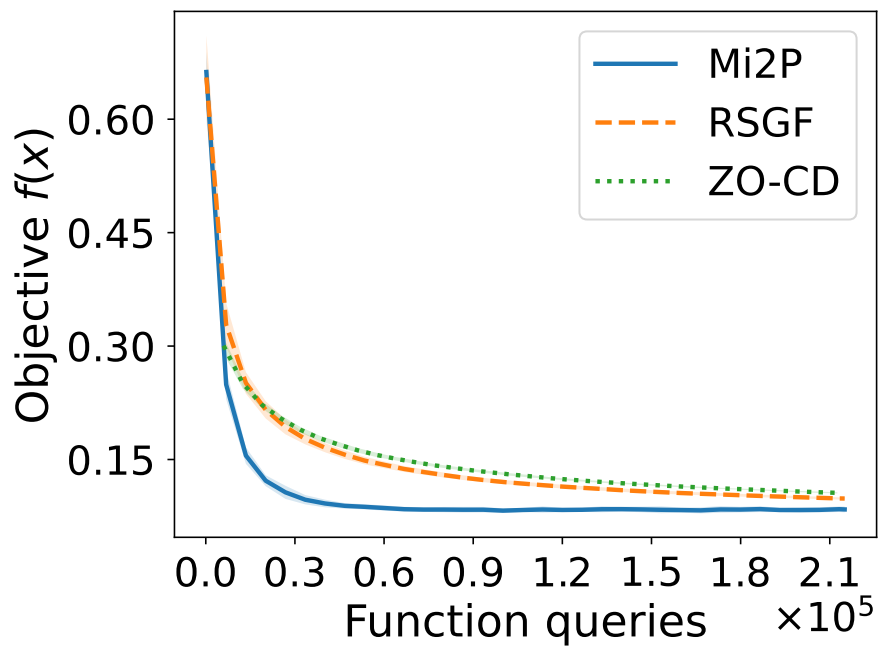}
    \caption{$b=100$}
  \end{subfigure}
  \caption{\textbf{Breast Cancer (logistic):} $f(\vx)$ vs.\ queries for different
  batch sizes. Mean $\pm$ one s.d.\ across 20 runs.}
  \label{fig:bc_all}
\end{figure}

\section{Limitations, Future Work, and Conclusion}

\paragraph{Limitations.}
Our analysis assumes $(L_0,L_1)$-smoothness and bounded variance. Extensions to
weaker smoothness or heavy-tailed noise remain open. In addition, while momentum
is highly effective in first-order methods, our negative result shows that its
direct adaptation fails for Random Search, leaving variance reduction reliant on
larger minibatches.

\paragraph{Future work.}
Promising directions include: (i) designing momentum-like schemes tailored to
function differences, (ii) exploring more memory-efficient variance reduction
methods, and (iii) evaluating Random Search in large-scale human-in-the-loop
settings such as online A/B testing or RLHF.

\paragraph{Conclusion.}
We gave a unified analysis of Random Search under weak average smoothness,
standard sample smoothness, and finite-sum objectives with variance reduction,
and extended the framework to inexact human/helper feedback. Our results clarify
the attainable rates of Random Search and highlight both its strengths and open
challenges.

\newpage
\bibliography{stp}


\clearpage
\appendix

\onecolumn
\setlength{\footskip}{24pt} 
\raggedbottom 


\aistatstitle{
Supplementary Materials}

\section{Preliminaries}\begin{propbox}
\begin{lemma}[Variance of a minibatch average under bounded variance]
\label{lem:var-average}
Let $(Z_i)_{i=1}^b$ be i.i.d.\ real-valued random variables with
$\E[Z_1]=0$ and $\Var(Z_1)\le \sigma^2<\infty$. Define the minibatch
average $\overline{Z} := \tfrac{1}{b}\sum_{i=1}^b Z_i$. Then
\[
    \E[\overline{Z}] = 0
    \qquad\text{and}\qquad
    \Var(\overline{Z}) = \frac{\Var(Z_1)}{b} \;\le\; \frac{\sigma^2}{b}.
\]
\end{lemma}
\end{propbox}
\begin{proof}
Linearity of expectation gives $\E[\overline{Z}]=\tfrac{1}{b}\sum_{i=1}^b \E[Z_i]=0$.
For the variance,
\[
    \Var(\overline{Z})
    = \Var\!\Big(\frac{1}{b}\sum_{i=1}^b Z_i\Big)
    = \frac{1}{b^2}\sum_{i=1}^b \Var(Z_i)
      + \frac{2}{b^2}\sum_{1\le i<j\le b}\Cov(Z_i,Z_j).
\]
Independence implies $\Cov(Z_i,Z_j)=0$ for $i\neq j$, so
$\Var(\overline{Z})=\tfrac{1}{b^2}\cdot b\,\Var(Z_1)=\Var(Z_1)/b \le \sigma^2/b$.
\end{proof}

\paragraph{Vector variant.}
If $(\vz_i)_{i=1}^b$ are i.i.d.\ $\R^d$-valued with $\E[\vz_1]=\vzero$ and
$\E\|\vz_1\|^2\le \sigma^2$, then for $\overline{\vz} := \tfrac{1}{b}\sum_{i=1}^b \vz_i$,
\[
  \E\big[\overline{\vz}\big]=\vzero
  \qquad\text{and}\qquad
  \E\|\overline{\vz}\|^2
  = \frac{1}{b}\,\E\|\vz_1\|^2 \;\le\; \frac{\sigma^2}{b}.
\]
The proof is identical, using $\E\|\sum_i \vz_i\|^2 = \sum_i \E\|\vz_i\|^2$
by independence and zero mean.

\medskip
\begin{propbox}
\begin{lemma}[Jensen's inequality]
\label{lem:jensen}
Let $\phi:\R\to\R$ be convex and let $X$ be an integrable real-valued random
variable. Then
\[
    \phi\big(\E[X]\big) \;\le\; \E\big[\phi(X)\big].
\]
In particular, for $\phi(x)=|x|$ we have $|\E[X]|\le \E[|X|]$.
\end{lemma}
\end{propbox}
\begin{proof}
By convexity, for any $x_0\in\R$ there exists a subgradient $g\in\partial\phi(x_0)$
such that $\phi(x)\ge \phi(x_0) + g\,(x-x_0)$ for all $x\in\R$. Taking $x_0=\E[X]$
and then expectations,
\[
  \E[\phi(X)] \;\ge\; \phi(\E[X]) + g\,\E[X-\E[X]] \;=\; \phi(\E[X]).
\]
For $\phi(x)=x^2$, convexity gives $\E[|X|]\le \sqrt{\E[X^2]}$ immediately.
\end{proof}

\begin{propbox}
 \begin{lemma}[Adaptive Smoothness Inequality] \label{lem:adaptive_smoothness} Consider a differentiable function $f: \R^d \to \mathbb{R}$ satisfying the following property for all $\vx, \vx^\prime \in \R^d$: $$ \big\|\nabla f(\vx) - \nabla f(\vx^\prime)\big\| \leq \min \big(\cL(\vx), \cL(\vx^\prime)\big) \|\vx - \vx^\prime\|, $$ for some function $\cL: \R^d \to \mathbb{R}^{+}$. Then, the function value is bounded by the quadratic approximation at $\vx$ as follows: $$ \big|f(\vx^\prime) - f(\vx) - \nabla f(\vx)^\top (\vx^\prime - \vx)\big| \leq \dfrac{\cL(\vx)}{2}\|\vx^\prime - \vx\|^2. $$ \end{lemma} 
 \end{propbox}
 \begin{proof} We begin with the fundamental theorem of calculus for the function $f$: \begin{align*} f(\vx^\prime) &= f(\vx) + \int_0^1 \nabla f(\vx + t(\vx^\prime - \vx))^\top (\vx^\prime - \vx) dt \\ &= f(\vx) + \nabla f(\vx)^\top (\vx^\prime - \vx) + \int_0^1 \Big[\nabla f(\vx + t(\vx^\prime - \vx)) - \nabla f(\vx)\Big]^\top (\vx^\prime - \vx) dt \end{align*} Rearranging the terms, we isolate the error term: $$ f(\vx^\prime) - f(\vx) - \nabla f(\vx)^\top (\vx^\prime - \vx) = \int_0^1 \Big[\nabla f(\vx + t(\vx^\prime - \vx)) - \nabla f(\vx)\Big]^\top (\vx^\prime - \vx) dt $$ Taking the absolute value and applying the generalized Cauchy-Schwarz inequality, followed by the adaptive smoothness property: \begin{align*} \big|f(\vx^\prime) - f(\vx) - \nabla f(\vx)^\top (\vx^\prime - \vx)\big| &\leq \Big|\int_0^1 \Big[\nabla f(\vx + t(\vx^\prime - \vx)) - \nabla f(\vx)\Big]^\top (\vx^\prime - \vx) dt\Big| \\ &\leq \int_0^1 \Big\|\nabla f(\vx + t(\vx^\prime - \vx)) - \nabla f(\vx)\Big\| \cdot \Big\|\vx^\prime - \vx\Big\| dt \end{align*} Using the adaptive smoothness property with $\vx_t = \vx + t(\vx^\prime - \vx)$ and $\vx$: $$ \big\|\nabla f(\vx_t) - \nabla f(\vx)\big\| \leq \min \big(\cL(\vx_t), \cL(\vx)\big) \|\vx_t - \vx\| $$ Since $\min (\cL(\vx_t), \cL(\vx)) \le \cL(\vx)$ and $\|\vx_t - \vx\| = \|t(\vx^\prime - \vx)\| = t \|\vx^\prime - \vx\|$, we continue the inequality: \begin{align*} \big|f(\vx^\prime) - f(\vx) - \nabla f(\vx)^\top (\vx^\prime - \vx)\big| &\leq \int_0^1 \Big[\cL(\vx) \cdot t \|\vx^\prime - \vx\|\Big] \cdot \|\vx^\prime - \vx\| dt \\ &\leq \cL(\vx) \|\vx^\prime - \vx\|^2 \int_0^1 t dt \\ &= \cL(\vx) \|\vx^\prime - \vx\|^2 \Big[\frac{t^2}{2}\Big]_0^1 \\ &= \dfrac{\cL(\vx)}{2}\|\vx^\prime - \vx\|^2. \end{align*} \end{proof}

\section{Analysis of the General Algorithm}

We consider the optimization problem
\[
    \min_{\vx\in\R^d} f(\vx).
\]

At iteration $t$, Algorithm~\ref{alg:srs} samples a random direction
$\vs_t \sim \cD$, defines two scalar quantities $M_t^\pm$, and updates
\[
    \vx_{t+1}
    =
    \begin{cases}
        \vx_t + \eta_t \vs_t, & \text{if } M_t^+ \le M_t^-, \\[4pt]
        \vx_t - \eta_t \vs_t, & \text{otherwise}.
    \end{cases}
\]
Equivalently,
\[
    \vx_{t+1}
    = \vx_t - \eta_t\,\operatorname{sign}(M_t^+ - M_t^-)\,\vs_t.
\]

We assume that $f$ is $(L_0,L_1)$-smooth, meaning that for all
$\vx,\vy\in\R^d$,
\begin{equation}
\label{eq:L0L1-smooth}
    \|\nabla f(\vx) - \nabla f(\vy)\|
    \le (L_0 + L_1 \|\nabla f(\vx)\|)\,\|\vx-\vy\|.
\end{equation}
By Lemma~\ref{lem:adaptive_smoothness}, this implies that for all
$\vx,\vy\in\R^d$,
\begin{equation}
\label{eq:smooth-upper-bound}
    f(\vy)
    \le
    f(\vx)
    + \langle \nabla f(\vx), \vy - \vx\rangle
    + \frac{L_0 + L_1 \|\nabla f(\vx)\|}{2}\,\|\vy-\vx\|^2.
\end{equation}
\begin{propbox}
\begin{theorem}[General Descent Inequality]
\label{thm:general-descent}
Let $f$ satisfy~\eqref{eq:L0L1-smooth} and let the random directions
$\vs_t\sim\cD$ satisfy Assumption~\ref{ass:distribution}, namely
\(
    \E\|\vs_t\|^2 = 1
\)
and
\(
    \E[\,|\langle\nabla f(\vx),\vs_t\rangle|\,] \ge \mu_{\cD}\|\nabla f(\vx)\|.
\)
Then, for any step size $\eta_t \le \mu_{\cD}/L_1$,
the update defined above satisfies
\begin{align}
\E[f(\vx_{t+1})]
&\le
\E[f(\vx_t)]
- \tfrac{\mu_{\cD}}{2}\,\eta_t\,\E\|\nabla f(\vx_t)\|
+ \tfrac{L_0}{2}\,\eta_t^2
+ 2\,\E\!\left[
    \big|
        M_t^+ - M_t^- - (f(\vx_t^+) - f(\vx_t^-))
    \big|
\right],
\label{eq:general-descent}
\end{align}
where $\vx_t^\pm = \vx_t \pm \eta_t \vs_t$.
\end{theorem}
\end{propbox}
\begin{proof}
Applying~\eqref{eq:smooth-upper-bound} with
$\vy = \vx_t^\pm = \vx_t \pm \eta_t \vs_t$ and $\vx = \vx_t$ gives
\begin{equation*}
    f(\vx_t^{\pm})
    \le
    f(\vx_t)
    \pm \eta_t \langle \nabla f(\vx_t), \vs_t\rangle
    + \frac{L_0 + L_1 \|\nabla f(\vx_t)\|}{2}\,\eta_t^2 \|\vs_t\|^2.
\end{equation*}
Since $M_t^\pm$ are estimates of $f(\vx_t^\pm)$, we have
\begin{equation*}
    M_t^\pm
    \le
    f(\vx_t)
    \pm \eta_t \langle \nabla f(\vx_t), \vs_t\rangle
    + \frac{L_0 + L_1 \|\nabla f(\vx_t)\|}{2}\,\eta_t^2\|\vs_t\|^2
    + |M_t^\pm - f(\vx_t^\pm)|.
\end{equation*}
Hence,
\begin{equation}
\label{eq:M-ineq}
    M_t^\pm
    \le
    f(\vx_t)
    \pm \eta_t \langle \nabla f(\vx_t), \vs_t\rangle
    + \tfrac{L_0 + L_1 \|\nabla f(\vx_t)\|}{2}\,\eta_t^2\|\vs_t\|^2
    + \sum_{i\in\{\pm\}} |M_t^i - f(\vx_t^i)|.
\end{equation}

Suppose $M_t^+ \le M_t^-$. Then $\vx_{t+1}=\vx_t^+$ and from
\eqref{eq:M-ineq},
\begin{align*}
    M_t^+
    &\le
    f(\vx_t)
    - \eta_t |\langle \nabla f(\vx_t), \vs_t\rangle|
    + \tfrac{L_0 + L_1 \|\nabla f(\vx_t)\|}{2}\,\eta_t^2\|\vs_t\|^2
    + \sum_{i\in\{\pm\}} |M_t^i - f(\vx_t^i)|.
\end{align*}
Since $f(\vx_{t+1}) = f(\vx_t^+)\le M_t^+ + |M_t^+ - f(\vx_t^+)|$, we obtain
\[
    f(\vx_{t+1})
    \le
    f(\vx_t)
    - \eta_t |\langle \nabla f(\vx_t), \vs_t\rangle|
    + \tfrac{L_0 + L_1 \|\nabla f(\vx_t)\|}{2}\,\eta_t^2\|\vs_t\|^2
    + 2 \sum_{i\in\{\pm\}} |M_t^i - f(\vx_t^i)|.
\]
The same bound holds symmetrically when $M_t^- < M_t^+$.

Taking expectations and using $\E\|\vs_t\|^2=1$ and
$\E[|\langle\nabla f(\vx_t),\vs_t\rangle|]\ge \mu_{\cD}\|\nabla f(\vx_t)\|$,
we get
\begin{align}
\E[f(\vx_{t+1})]
&\le
\E[f(\vx_t)]
- \eta_t \mu_{\cD}\E\|\nabla f(\vx_t)\|
+ \tfrac{L_0 + L_1\E\|\nabla f(\vx_t)\|}{2}\eta_t^2
+ 2\sum_{i\in\{\pm\}}\E|M_t^i - f(\vx_t^i)|.
\label{eq:pre-translation}
\end{align}

By the translation symmetry argument, shifting both $M_t^+$ and $M_t^-$ by any
constant $C_t$ does not change the update rule. Thus, we can replace the last
term by its optimal shift:
\[
    \min_{C_t}
    \sum_{i\in\{\pm\}} |M_t^i - C_t - f(\vx_t^i)|
    = |M_t^+ - M_t^- - (f(\vx_t^+) - f(\vx_t^-))|.
\]
Substituting into~\eqref{eq:pre-translation} and rearranging yields
\begin{align*}
\E[f(\vx_{t+1})]
&\le
\E[f(\vx_t)]
- \eta_t\big(\mu_{\cD} - \tfrac{L_1}{2}\eta_t\big)\E\|\nabla f(\vx_t)\|
+ \tfrac{L_0}{2}\eta_t^2
+ 2\E\big|M_t^+ - M_t^- - (f(\vx_t^+) - f(\vx_t^-))\big|.
\end{align*}
Finally, since $\eta_t\le\mu_{\cD}/L_1$, we have
$\mu_{\cD}-\tfrac{L_1}{2}\eta_t\ge \mu_{\cD}/2$, giving the claimed inequality
\eqref{eq:general-descent}.
\end{proof}

\section{General Stochastic Functions with Subsampling}

In this section, we consider the case where we define:
\[
   M_t^\pm := \frac{1}{b} \sum_{j=1}^b f_{\xi^t_j}(\vx_t^\pm),
   \qquad \{\xi^t_j\}_{j=1}^b \ \text{i.i.d.}
\]

\subsection{Analysis under Average Smoothness}

Suppose that $f$ satisfies the $(L_0,L_1)$-smoothness condition
\eqref{eq:L0L1-smooth} and that the random directions $\vs_t \sim \cD$
satisfy Assumption~\ref{ass:distribution}, i.e.
$\E[\|\vs_t\|^2]=1$ and
$\E[|\langle \nabla f(\vx_t),\vs_t\rangle|] \ge \mu_\cD \|\nabla f(\vx_t)\|$.

Then, for any $\eta_t=\eta\le\mu_\cD/L_1$, we have
\begin{align}
\E[f(\vx_{t+1})]
&\le
\E[f(\vx_t)]
- \tfrac{\mu_{\cD}}{2}\,\eta\,\E\|\nabla f(\vx_t)\|
+ \tfrac{L_0}{2}\,\eta^2
+ 2\,\E\!\left[
    \big|
        M_t^+ - M_t^- - (f(\vx_t^+) - f(\vx_t^-))
    \big|
\right].
\label{eq:avg-smoothness-initial}
\end{align}
Let $X_j = f_{\xi^t_j}(\vx_t^+) - f_{\xi^t_j}(\vx_t^-) - (f(\vx_t^+) - f(\vx_t^-))$.
Then $\E[X_j]=0$ and suppose $\E[X_j^2]\le\sigma_0^2 < \infty$.

Using the independence of $\xi^t_j$ and applying Jensen's
inequality and Lemma~\ref{lem:var-average}, it holds that
\begin{align}
\E\!\left[
    \big|
        M_t^+ - M_t^- - (f(\vx_t^+) - f(\vx_t^-))
    \big|
\right]
&= \E\!\left[\Big|\frac{1}{b}\sum_{j=1}^b X_j\Big|\right]
\le \sqrt{\E\!\left[\Big(\frac{1}{b}\sum_{j=1}^b X_j\Big)^2\right]}
\le \frac{\sigma_0}{\sqrt{b}}.
\label{eq:avg-smoothness-variance}
\end{align}
Substituting~\eqref{eq:avg-smoothness-variance} into
\eqref{eq:avg-smoothness-initial} yields
\begin{align}
\E[f(\vx_{t+1})]
&\le
\E[f(\vx_t)]
- \tfrac{\mu_{\cD}}{2}\,\eta\,\E\|\nabla f(\vx_t)\|
+ \tfrac{L_0}{2}\,\eta^2
+ \frac{2\sigma_0}{\sqrt{b}}.
\label{eq:avg-smoothness-recursion}
\end{align}
Averaging this inequality over $t=0,\dots,T-1$ and denoting
$F_0 = f(\vx_0) - f^\star$, with $f^\star = \inf_\vx f(\vx) > -\infty$, we obtain
\begin{equation}
\label{eq:avg-smoothness-gradavg}
\frac{1}{2T}\sum_{t=0}^{T-1} \E\|\nabla f(\vx_t)\|
\;\le\;
\frac{1}{\mu_\cD}
\left(
    \frac{F_0}{\eta T}
    + \frac{L_0}{2}\eta
    + \frac{2\sigma_0}{\eta\sqrt{b}}
\right).
\end{equation}

\paragraph{Choice of stepsize.}
We choose the stepsize $\eta$ minimizing the right-hand side of
\eqref{eq:avg-smoothness-gradavg} subject to
$\eta \le \tfrac{\mu_\cD}{L_1}$, namely
\begin{equation}
\label{eq:avg-smoothness-eta}
\eta
= \min\!\left(
    \tfrac{\mu_\cD}{L_1},\,
    \sqrt{\tfrac{F_0}{L_0 T}},\,
    \sqrt{\tfrac{\sigma_0}{L_0\sqrt{b}}}
\right).
\end{equation}
Substituting this choice into~\eqref{eq:avg-smoothness-gradavg} gives
\begin{equation}
\label{eq:avg-smoothness-final}
\frac{1}{2T}\sum_{t=0}^{T-1} \E\|\nabla f(\vx_t)\|
\;=\;
\frac{1}{\mu_\cD}\,
\cO\!\left(
    \sqrt{\tfrac{L_0F_0}{T}}
    + \tfrac{L_1 F_0}{\mu_\cD T}
    + \tfrac{\sqrt{\sigma_0}}{b^{1/4}}
\right).
\end{equation}

\paragraph{Complexity.}
To ensure
$\tfrac{1}{2T}\sum_{t=0}^{T-1}\E\|\nabla f(\vx_t)\| \le \varepsilon$,
we require
\[
    T
    = \cO\!\left(
        \tfrac{L_0 F_0}{\mu_\cD^2 \varepsilon^2}
        + \tfrac{L_1 F_0}{\mu_\cD^2 \varepsilon}
    \right),
    \qquad
    b
    = \cO\!\left(
        \tfrac{\sigma_0^2}{\mu_\cD^4 \varepsilon^4}
    \right).
\]
Since $\mu_\cD = \cO(d^{-1/2})$ for isotropic distributions, this simplifies to
\[
    T
    = \cO\!\left(
        \tfrac{d L_0 F_0}{\varepsilon^2}
        + \tfrac{d L_1 F_0}{\varepsilon}
    \right),
    \qquad
    b
    = \cO\!\left(
        \tfrac{d^2 \sigma_0^2}{\varepsilon^4}
    \right).
\]
The total number of function evaluations (two per minibatch per iteration)
is therefore of order
\begin{equation}
\label{eq:avg-smoothness-complexity}
T \times b
\;=\;
\cO\!\left(
    d^3\sigma_0^2 F_0
    \Big(
        \tfrac{L_0}{\varepsilon^6}
        + \tfrac{L_1}{\varepsilon^5}
    \Big)
\right).
\end{equation}

\subsection{Analysis under individual Smoothness}\label{AppSec:IndSm}
For simplicity, in section~4.2, we assumed that each \(f_{\xi}\) is \((L_0,L_1)\)-smooth with respect to \(\nabla f\).
Here we show that the same improved rate still holds when each \(f_{\xi}\) is instead \((L_0,L_1)\)-smooth with respect to its own gradient \(\nabla f_{\xi}\).

Assume that we work under the following assumptions:

\begin{assumption}\label{ass:smooth_ind}
For $\mathcal{P}$-a.e.\ $\xi$, the mapping $x \mapsto f_{\xi}(\vx)$ is 
$(L_0, L_1)$-smooth.
\end{assumption}

\begin{assumption}[Bounded gradient variance]\label{ass:var_bound}
There exists $\sigma \ge 0$ such that for all $\vx \in \mathbb{R}^d$,
\[
  \mathbb{E}_\xi\!\big[\|\nabla f_{\xi}(\vx)-\nabla f(\vx)\|^2\big]\ \le\ \sigma^2 .
\]
\end{assumption}

\begin{propbox}
\begin{lemma}[Equivalent second-moment form]\label{lem:second_moment_equiv}
Assumption~\ref{ass:var_bound} holds if and only if, for all $\vx \in \mathbb{R}^d$,
\[
  \mathbb{E}_\xi\!\big[\|\nabla f_{\xi}(\vx)\|^2\big]\ \le\ \|\nabla f(\vx)\|^2 + \sigma^2 .
\]
\end{lemma}
\end{propbox}
\begin{proof}
Let $G_\xi(\vx):=\nabla f_\xi(\vx)$ and $G(\vx):=\nabla f(\vx)=\mathbb{E}_\xi[G_\xi(\vx)]$.
The variance decomposition identity gives
\[
  \mathbb{E}_\xi\!\big[\|G_\xi(\vx)\|^2\big]
  \;=\; \|G(\vx)\|^2 + \mathbb{E}_\xi\!\big[\|G_\xi(\vx)-G(\vx)\|^2\big].
\]
Therefore
\[
  \mathbb{E}_\xi\!\big[\|G_\xi(\vx)-G(\vx)\|^2\big]\le \sigma^2
  \quad\Longleftrightarrow\quad
  \mathbb{E}_\xi\!\big[\|G_\xi(\vx)\|^2\big]\le \|G(\vx)\|^2+\sigma^2.
\]
\end{proof}

\begin{propbox}
\begin{lemma}\label{Lemma:sampling}
By sampling independent samples $\xi_1,\ldots,\xi_b$ from $\mathcal{P},$ we have for all $x,y \in \mathbb{R}^d:$    $$\mathbb{E}\lVert \frac{1}{b}\sum_{j=1}^b \nabla f_{\xi_j}(y) -\frac{1}{b}\sum_{j=1}^b \nabla f_{\xi_j}(x)\rVert\le  \big( L_0+L_1\sigma +L_1 \| \nabla f(x)\| \big)\,\|x-y \|,$$
and we have also:
 $$\lVert \nabla f(y) - \nabla f(x)\rVert\le \big( L_0+L_1\sigma +L_1 \| \nabla f(x)\|\big)\, \|x-y \|.$$
\end{lemma}
\end{propbox}
\begin{proof}
Let $x,y \in \mathbb{R}^d$. Using Assumption~\ref{ass:smooth_ind}, for $\mathbb{P}$-almost every $\xi$, we have:
\begin{equation*}
\|\nabla f_{\xi}(y)-\nabla f_{\xi}(x)\|
\;\le\;
\big(L_0+L_1\|\nabla f_{\xi}(x)\|\big)\,\|y-x\|.
\end{equation*}
By sampling independent samples $\xi_1,\ldots,\xi_b$ from $\mathcal{P},$ we have:
\begin{align*}
\lVert \frac{1}{b}\sum_{j=1}^b \nabla f_{\xi_j}(y) -\frac{1}{b}\sum_{j=1}^b \nabla f_{\xi_j}(x)\rVert&\le \frac{\sum_{j=1}^b  \lVert \nabla f_{\xi_j}(y) - \nabla f_{\xi_j}(x)\rVert}{b} \\
&\le \bigg( L_0+L_1 \frac{ \sum_{j=1}\|\nabla f_{\xi_j}(x)\| }{b} \bigg) \|x-y \|.
\end{align*}
This implies that:
\begin{align*}
\mathbb{E}\lVert \frac{1}{b}\sum_{j=1}^b \nabla f_{\xi_j}(y) -\frac{1}{b}\sum_{j=1}^b \nabla f_{\xi_j}(x)\rVert &\le \big(L_0+L_1\mathbb{E}\|\nabla f_{\xi_1}(x)\|\big) \|x-y \|\\
&\le L_0+L_1\sqrt{\mathbb{E} \big(\|\nabla f_{\xi_1}(x)\|^2}\big)  .\, \|x-y \|\\
&\le \big(L_0+L_1\sqrt{\sigma^2+ \|\nabla f(x)\|^2 } \big) \|x-y \|\\
&\le \big( L_0+L_1\sigma +L_1 \| \nabla f(x)\|\big)\|x-y \|.
\end{align*}
For the second claim,
using Jensen’s inequality and assumption \ref{ass:var_bound}, we obtain:
\begin{align*}
\|\nabla f(y)-\nabla f(x)\|
&= \big\|\mathbb{E}_{\xi}[\nabla f_{\xi}(y)-\nabla f_{\xi}(x)]\big\|  \\
&\le \mathbb{E}_{\xi}\big[\|\nabla f_{\xi}(y)-\nabla f_{\xi}(x)\|\big] \\
&\le \mathbb{E}_{\xi}\!\Big[L_0+L_1\|\nabla f_{\xi}(x)\|\  \,\Big] \|x-y \|  \\
&\le \Big[L_0+L_1 \mathbb{E}\|\nabla f_{\xi}(x)\|  \Big] \|x-y \| \\
&\le\!\left(L_0 + L_1\sqrt{\sigma^2+\|\nabla f(x)\|^2}\right) \|x-y \| \\
&\le \!\left(L_0 +L_1\sigma + L_1\|\nabla f(x)\|\right) \|x-y \|.
\end{align*}

\end{proof}

For simplicity, we assume throughout the rest of this section that the search distribution~$\mathcal{D}$ is uniform on the unit sphere.
\begin{propbox}
\begin{lemma}\label{lem:uuT}
Let $u\in\mathbb{S}^{d-1}$ be uniformly distributed on the unit sphere
and let $v\in\mathbb{R}^d$. Then
\[
\mathbb{E}\big[\langle v,u\rangle^2\big] \;=\; \frac{\|v\|^2}{d}.
\]

\end{lemma}
\end{propbox}

\begin{proof}
Choose $Q$ an orthogonal matrix so that
$Qv=\|v\|e_1$, where $e_1=(1,0,\dots,0)^\top$. Then
$$
\mathbb{E}\big[\langle v,u\rangle^2\big]
= \mathbb{E}\big[\langle Qv,Qu\rangle^2\big]
= \|v\|^2\,\mathbb{E}\big[\langle e_1,u\rangle^2\big]
= \|v\|^2\,\mathbb{E}[u_1^2]=\frac{||v||^2}{d}.$$
\end{proof}

By translation invariance (Section~3.5), we may shift $M_t^\pm$ by an optimal constant $c_t$ so that we are left with an estimation error given by:
\[
\big| (M_t^+ - M_t^-) - ( f(\vx_t^{+}) - f(\vx_t^{-})) \big|.
\]

Under the assumptions stated above, the following lemma provides an explicit upper bound on the expected value of this error.
\begin{propbox}
\begin{lemma}\label{lem:diff-exp-iso}
Let $\vx_t^\pm := \vx_t \pm \eta s_t$, and assume Assumptions~\ref{ass:smooth_ind} and~\ref{ass:var_bound} hold. 
Suppose the search direction $s_t$ is drawn uniformly from the unit sphere $\mathbb{S}^{d-1}$. 
Let $\xi_{t,1},\ldots,\xi_{t,b}$ be i.i.d.\ samples from $\mathcal{P}$, and define the mini-batch estimators
\[
M_t^\pm \;:=\; \frac{1}{b}\sum_{j=1}^b f_{\xi_{t,j}}(\vx_t^\pm).
\]
Then
$$\mathbb{E}\big[\big|(M^+-M^-) - (f(x_t^+)-f(x_t^-))\big|  \, \big| x_t\big]\le\frac{2\eta \sigma}{\sqrt{db}}+ 2\eta^2\big(L_0+L_1\sigma + L_1\| \nabla f(x_t) \| \big).$$
\end{lemma}
\end{propbox}

\begin{proof}
Denote: $x_t^\pm=x_t\pm\eta s_t$.  By the fundamental theorem of calculus, we have:
$$\begin{cases}
f_\xi(x_t^+)-f_\xi(x_t)
\;=\;\int_0^{1}\!\!\!\big\langle \nabla f_\xi(x_t+u \eta s_t),\,\eta s_t\big\rangle\,du,
\\
f_\xi(x_t^-)-f_\xi(x_t)
\;=\;\int_0^{-1}\!\!\!\big\langle \nabla f_\xi(x_t+u \eta s_t),\,\eta s_t\big\rangle\,du.
\end{cases} \Longrightarrow f_\xi(x_t^+)-f_\xi(x_t^-)= \int_{-1}^{1}\!\!\!\big\langle \nabla f_\xi(x_t+u \eta s_t),\,\eta s_t\big\rangle\,du.
$$
Averaging over the minibatch and subtracting the population identity gives:
\begin{align*}
(M^+-M^-) - (f(x_t^+)-f(x_t^-))
&= \eta\int_{-1}^{1}\!\Big\langle \underbrace{\tfrac1b\sum_{j=1}^b\nabla f_{\xi_j}(x_t+u \eta s_t)-\nabla f(x_t+u \eta s_t)}_{=:~\Delta(x_t,s_t)},\,s_t\Big\rangle\,du. \tag{$\star$}
\end{align*}
\noindent
Denote:
$\overline g(x):=\frac1b\sum_{j=1}^b\nabla f_{\xi_j}(x)$. We have:
\[
\Delta(x_t,s_t)
=\big(\overline g(x_t)-\nabla f(x_t)\big)
+\big(\overline g(x_t+u \eta s_t)-\overline g(x_t)\big)
-\big(\nabla f(x_t+u \eta s_t)-\nabla f(x_t)\big).
\]
Plugging this into $(\star)$ and using the triangle inequality yields
\begin{align*}
\big|(M^+-M^-) - (f(x_t^+)-f(x_t^-))\big|
&\le 2\eta\,\big|\langle \overline g(x_t)-\nabla f(x_t),\,s_t\rangle\big| \\
&\qquad + \eta\!\int_{-1}^{1}\!\Big(\big|\langle \overline g(x_t+u \eta s_t)-\overline g(x_t),s_t\rangle\big|
\\
&\qquad +\big|\langle \nabla f(x_t+u \eta s_t)-\nabla f(x_t),s_t\rangle\big|\Big)\,du. \tag{$\dagger$}
\end{align*}

By denoting $y_{u,t}=x_t+u \eta s_t.$  By lemma \ref{Lemma:sampling}, we have:
\begin{equation}\label{eq:grad-diff-bounds}
\begin{cases}
\displaystyle
\big|\langle \nabla f(y_{u,t}) - \nabla f(x_t), s_t \rangle\big|
~\le~
\|\nabla f(y_{u,t}) - \nabla f(x_t)\|\,\|s_t\|
~\le~
\big(L_0 + L_1\sigma + L_1\|\nabla f(x_t)\|\big)\,|u|\eta\,\|s_t\|^2,
&\hspace{-1em}\\[1.2em]
\displaystyle
\mathbb{E}\!\left[\big|\langle \overline g(y_{u,t}) - \overline g(x_t), s_t \rangle\big| \,\middle|\, x_t, s_t\right]
~\le~
\mathbb{E}\!\left[\|\overline g(y_{u,t}) - \overline g(x_t)\| \,\middle|\, x_t, s_t\right]\!\|s_t\|
~\le~
\big(L_0 + L_1\sigma + L_1\|\nabla f(x_t)\|\big)\,|u|\eta\,\|s_t\|^2,
&\hspace{-1em}
\end{cases}
\end{equation}
Using  Cauchy–Schwarz inequality and lemma \ref{lem:uuT}  we have:

\begin{align*}
    \mathbb{E}[ \big|\langle \overline g(x_t)-\nabla f(x_t),\,s_t\rangle\big|   \, |x_t,\xi_1,\ldots,\xi_b  ] &\le \sqrt{ \mathbb{E}[ \langle \overline g(x_t)-\nabla f(x_t),\,s_t\rangle^2   \, |x_t,\xi_1,\ldots,\xi_b   ]  }\\
    &=\frac{\|\overline g(x_t)-\nabla f(x_t)\| }{\sqrt{d}} 
\end{align*}
We  have also:  $\mathbb{E}\big[\|\overline g(x_t)-\nabla f(x_t)\| \, | x_t \big]\le \sqrt{\mathbb{E}\big[ \|\overline g(x_t)-\nabla f(x_t)\|^2 \, | x_t  \big] }=\sqrt{\frac{1}{b} \mathbb{E}\big[ \|\nabla f_{\xi_1}(x_t)-\nabla f(x_t)\|^2 \, | x_t  \big]  }.$ This implies that: $\mathbb{E}\big[\|\overline g(x_t)-\nabla f(x_t)\| \, | x_t \big]\le \frac{\sigma}{\sqrt{b}}.$ It follows that: $\mathbb{E}\big[\frac{\|\overline g(x_t)-\nabla f(x_t)\| }{\sqrt{d}}  \, | x_t\big]\le \frac{\sigma}{\sqrt{db}}.$ Therefore:
$$ \mathbb{E}\big[ \big|\langle \overline g(x_t)-\nabla f(x_t),\,s_t\rangle\big|\, | x_t   \big]\le  \frac{\sigma}{\sqrt{db}}.$$

Combining the bounds in \ref{eq:grad-diff-bounds} with ($\dagger$) and the inequality above, we obtain:
$$\mathbb{E}\big[\big|(M^+-M^-) - (f(x_t^+)-f(x_t^-))\big|  \, \big| x_t\big]\le\frac{2\eta \sigma}{\sqrt{db}}+ 2\eta^2\big(L_0+L_1\sigma + L_1\| \nabla f(x_t) \| \big).$$
\end{proof}

\begin{propbox}

\begin{theorem}\label{lem:indiv-descent}
Assume that assumptions \ref{ass:smooth_ind}, and \ref{ass:var_bound} hold, the search directions sampled from the uniform distribution over the unit sphere, and the step size $
\eta \le \; \frac{\mu_{\mathcal{D}}}{16L_1}
$. We have for all $t \ge 0$:
\[
\frac{1}{T}\sum_{t=0}^{T-1}\mathbb{E}\big\|\nabla f(\vx_t)\big\|
\;\le\;
\frac{4F_0}{\mu_{\mathcal{D}}\eta\,T}
\;+\;\frac{16}{\mu_{\mathcal{D}}}\,\frac{\sigma}{\sqrt{db}}
\;+\;\frac{18\eta}{\mu_{\mathcal{D}}}\,(L_0+L_1\sigma).
\]
where $F_0:=f(\vx_0)-f^\star.$
\end{theorem}
\end{propbox}

\begin{proof}[Proof of Theorem \ref{lem:indiv-descent}]
Let $t\ge 0$. Using lemma \ref{Lemma:sampling},
   $f$ is $(L_0+L_1\sigma,L_1)$ smooth, and by applying \cref{thm:general-descent} , since $\eta\le\mu_\cD/L_1$, we have:
$$
\E[f(\vx_{t+1})]
\le
\E[f(\vx_t)]
- \tfrac{\mu_{\cD}}{2}\,\eta\,\E\|\nabla f(\vx_t)\|
+ \tfrac{L_0+L_1\sigma}{2}\,\eta^2
+ 2\,\E\!\left[
    \big|
        M_t^+ - M_t^- - (f(\vx_t^+) - f(\vx_t^-))
    \big|
\right].    
      $$

   Using \cref{lem:diff-exp-iso}, we deduce that:
\begin{align*}
\mathbb{E}\big[f(\vx_{t+1})\big]
&\le
\mathbb{E}\big[f(\vx_t)\big]
- \tfrac{\mu_{\mathcal{D}}}{2}\,\eta\,\mathbb{E}\big\|\nabla f(\vx_t)\big\|
+ \tfrac{L_0+L_1\sigma}{2}\,\eta^2
+ \frac{4\eta}{\sqrt{db}}\,\sigma
+ 4\eta^2\Big(L_0+L_1\sigma+L_1\,\mathbb{E}\big\|\nabla f(\vx_t)\big\|\Big)\\
&=
\mathbb{E}\big[f(\vx_t)\big]
-\Big(\tfrac{\mu_{\mathcal{D}}}{2}\,\eta-4\eta^2 L_1\Big)\mathbb{E}\big\|\nabla f(\vx_t)\big\|
+ \underbrace{\Big(\tfrac{L_0+L_1 \sigma}{2}\,\eta^2+\tfrac{4\eta}{\sqrt{db}}\,\sigma+4\eta^2(L_0+L_1\sigma)\Big)}_{=:~\beta(\eta,b)}.
\end{align*}
Let
$
\alpha(\eta,b)\;:=\;\tfrac{\mu_{\mathcal{D}}}{2}\,\eta-4\eta^2 L_1.$
If we choose:
$
\eta \;\le\;\frac{\mu_{\mathcal{D}}}{16L_1}
$,
then $\alpha(\eta,b)\ge \tfrac{\mu_{\mathcal{D}}}{4}\,\eta$. Summing the one–step bound over $t=0,\ldots,T-1$ and telescoping gives:
\[
\sum_{t=0}^{T-1}\alpha(\eta,b)\,\mathbb{E}\big\|\nabla f(\vx_t)\big\|
\;\le\; \mathbb{E}\big[f(\vx_0)\big]-\mathbb{E}\big[f(\vx_T)\big] + T\,\beta(\eta,b)
\;\le\; F_0 + T\,\beta(\eta,b),
\]
where $F_0:=f(\vx_0)-f^\star$. Dividing by $T$ and by $\alpha(\eta,b)\ge \frac{\mu_{\mathcal{D}}}{4}\eta$ yields:
\[
\frac{1}{T}\sum_{t=0}^{T-1}\mathbb{E}\big\|\nabla f(\vx_t)\big\|
\;\le\;
\frac{4F_0}{\mu_{\mathcal{D}}\eta\,T}
\;+\;\frac{4}{\mu_{\mathcal{D}}\eta}\,\beta(\eta,b).
\]
Expanding $\beta(\eta,b)$ and simplifying, we obtain:
\[
\frac{1}{T}\sum_{t=0}^{T-1}\mathbb{E}\big\|\nabla f(\vx_t)\big\|
\;\le\;
\frac{4F_0}{\mu_{\mathcal{D}}\eta\,T}
\;+\;\frac{2L_0+2L_1 \sigma}{\mu_{\mathcal{D}}}\,\eta
\;+\;\frac{16}{\mu_{\mathcal{D}}}\,\frac{\sigma}{\sqrt{db}}
\;+\;\frac{16\eta}{\mu_{\mathcal{D}}}\,(L_0+L_1\sigma).
\]
\end{proof}
\paragraph{Choice of stepsize.}
We choose the stepsize $\eta$ minimizing the right-hand side of
Theorem~\ref{lem:indiv-descent} subject to
$\eta \le \tfrac{\mu_{\mathcal{D}}}{16 L_1}$, namely
\begin{equation}
\label{eq:indiv-eta}
\eta
\;=\;
\min\!\left(
    \tfrac{\mu_{\mathcal{D}}}{16 L_1},\,
    \sqrt{\tfrac{F_0}{(L_0+L_1\sigma)\,T}}
\right).
\end{equation}
Substituting this choice into Theorem~\ref{lem:indiv-descent} and simplifying yields
\begin{equation}
\label{eq:indiv-final}
\frac{1}{T}\sum_{t=0}^{T-1} \mathbb{E}\|\nabla f(\vx_t)\|
\;=\;
\frac{1}{\mu_{\mathcal{D}}}\,
\mathcal{O}\!\left(
    \frac{L_1 F_0}{\mu_{\mathcal{D}} T}
    + \sqrt{\tfrac{(L_0+L_1\sigma)\,F_0}{T}}
    + \frac{\sigma}{\sqrt{d\, b}}
\right).
\end{equation}

\paragraph{Complexity.}
For the uniform distribution over the unit sphere, we have
$\mu_{\mathcal{D}}\asymp d^{-1/2}$. With
$T=\Theta\!\big(dL_1/\varepsilon + d(L_0+L_1\sigma)F_0/\varepsilon^2\big)$ and 
$b=\Theta(\sigma^2/\varepsilon^2)$, we ensure
$\frac{1}{T}\sum_{t=0}^{T-1} \mathbb{E}\|\nabla f(\vx_t)\|\le \varepsilon$.
The total number of function evaluations is then
\[
\textsc{Calls}=2bT
\;=\;\mathcal{O}\!\left(\frac{d\,(L_0+L_1\sigma)\,F_0\,\sigma^2}{\varepsilon^4}\right).
\]

\color{black}

\section{Finite sum case and Variance reduction}

We now consider the finite-sum setting under stronger regularity assumptions.  
Specifically, assume that $f = \tfrac{1}{n}\sum_{i=1}^n f_i$ is $(L_0,L_1)$-smooth and that each component
function $f_i$ is $G$-Lipschitz, i.e.
\begin{equation}
\label{eq:G-Lipschitz}
    |f_i(\vx) - f_i(\vy)| \;\le\; G\,\|\vx-\vy\|,
    \qquad \forall\, \vx,\vy\in\R^d.
\end{equation}

We define
\[
   M_t^\pm =
   \begin{cases}
      f(\vx_t^\pm), & \text{if } t \equiv 0 \pmod{m}, \\[6pt]
      \tfrac{1}{b}\sum_{j=1}^b f_{\xi_j}(\vx_t^\pm),
      & \text{otherwise},
   \end{cases}
\]
\begin{propbox}
\begin{theorem}[Convergence under strong assumptions]
\label{thm:strong-assumptions}
Under Assumption~\eqref{eq:G-Lipschitz}, suppose that
$f$ is $(L_0,L_1)$-smooth, $\vs_t\sim\cD$ satisfies
Assumption~\ref{ass:distribution}, and a variance-reduction scheme with
snapshot period $m$ is used.  
Let $\eta_t\equiv\eta\le \mu_\cD/L_1$ and
$F_0:=f(\vx_0)-f^\star<\infty$. Then for each iteration,
\begin{equation}
\label{eq:strong-assumptions-one-step}
\E[f(\vx_{t+1})]
\;\le\;
\E[f(\vx_t)]
\;-\;\frac{\mu_{\cD}}{2}\,\eta\,\E\|\nabla f(\vx_t)\|
\;+\;\frac{L_0}{2}\,\eta^2
\;+\;\frac{4\,G\,m}{\sqrt{b}}.
\end{equation}
Averaging over $t=0,\dots,T-1$, we obtain
\begin{equation}
\label{eq:strong-assumptions-avg}
\frac{1}{T}\sum_{t=0}^{T-1}\E\|\nabla f(\vx_t)\|
\;\le\;
\frac{1}{\mu_{\cD}}
\left(
  \frac{F_0}{\eta T}
  + \frac{L_0}{2}\,\eta
  + \frac{4\,G\,m}{\sqrt{b}}
\right).
\end{equation}
\end{theorem}
\end{propbox}
Optimizing with respect to $\eta$ yields
\begin{equation}
\label{eq:strong-assumptions-optimized}
\frac{1}{T}\sum_{t=0}^{T-1}\E\|\nabla f(\vx_t)\|
\;=\;
\frac{1}{\mu_{\cD}}\,
\cO\!\left(
  \frac{L_1 F_0}{\mu_{\cD} T}
  + \sqrt{\frac{L_0 F_0}{T}}
  + \frac{G\,m}{\sqrt{b}}
\right).
\end{equation}
To guarantee
$\tfrac{1}{T}\sum_{t=0}^{T-1}\E\|\nabla f(\vx_t)\|\le\varepsilon$, it suffices to choose
\[
T = d F_0\,\cO\!\left(\frac{L_1}{\varepsilon} + \frac{L_0}{\varepsilon^2}\right),
\qquad
b=b(\varepsilon,m)=\cO\!\left(\frac{d G^2 m^2}{\varepsilon^2}\right).
\]
The total number of function evaluations is then
\begin{align}
\textsc{Calls}(m)
&= \frac{T}{m}\big(n + (m-1)b(m)\big) \notag\\
&= \cO\!\left(
    d F_0
    \Big(\frac{L_1}{\varepsilon} + \frac{L_0}{\varepsilon^2}\Big)
    \Big(\frac{n}{m} + \frac{d G^2 m^2}{\varepsilon^2}\Big)
\right).
\label{eq:strong-assumptions-calls}
\end{align}
Minimizing over $m$ with $b(\varepsilon,m)\le n$ gives
\[
    m^\star
    =
    \min\!\left\{
      \Big(\tfrac{n\varepsilon^2}{d G^2}\Big)^{1/3},
      \Big(\tfrac{n\varepsilon^2}{d G^2}\Big)^{1/2}
    \right\},
\]
and substituting $m^\star$ into~\eqref{eq:strong-assumptions-calls} yields the
optimal complexity
\begin{equation}
\label{eq:strong-assumptions-final}
\textsc{Calls}^\star
= \cO\!\left(
  \min\!\left\{
    \frac{d^{4/3} n^{2/3} G^{2/3}}{\varepsilon^{2/3}},
    d n
  \right\}
  F_0
  \Big(\frac{L_1}{\varepsilon} + \frac{L_0}{\varepsilon^2}\Big)
\right).
\end{equation}

\eqref{eq:strong-assumptions-final} is better than the deterministic complexity whenever $n\geq \tfrac{d G^2}{\varepsilon^2}.$ The quantity $d\,G^2$ can be interpreted as a noise level.

\begin{proof}
Define
\[
X_j := f_{\xi^t_j}(\vx_t^+) - f_{\xi^t_j}(\vx_t^-) - (f(\vx_t^+) - f(\vx_t^-)).
\]
By construction $\E[X_j]=0$.  
Let $\tilde{\vx}_t$ denote the snapshot point at the start of the current
variance-reduction block of length $m$. Using the $G$-Lipschitz property of each
$f_i$, we can write
\begin{align*}
\E[X_j^2]
&= \E\Big[
  \big|
    f_{\xi^t_j}(\vx_t^+) - f(\vx_t^+)
    - \big(f_{\xi^t_j}(\vx_t^-) - f(\vx_t^-)\big)
  \big|^2
\Big] \\
&= \E\Big[
  \big|
    (f_{\xi^t_j}(\vx_t^+) - f_{\xi^t_j}(\tilde{\vx}_t))
    - (f_{\xi^t_j}(\vx_t^-) - f_{\xi^t_j}(\tilde{\vx}_t))
    - \big(f(\vx_t^+) - f(\tilde{\vx}_t)\big)
    + \big(f(\vx_t^-) - f(\tilde{\vx}_t)\big)
  \big|^2
\Big] \\
&\le 2(2G)^2\big(
  \E\|\vx_t^+ - \tilde{\vx}_t\|^2
  + \E\|\vx_t^- - \tilde{\vx}_t\|^2
\big)
\;\le\; (4 G m \eta)^2,
\end{align*}
where the last inequality follows since the iterates remain within a distance
$\cO(m\eta)$ of the last snapshot over the $m$ local steps.  
Translation symmetry implies that this snapshot correction can be inserted
without explicitly modifying the algorithm.

By independence of the samples $\xi^t_j$ and
applying Lemma~\ref{lem:var-average} and Jensen's inequality, we have
\begin{align}
\E\!\left[
    \big|
        M_t^+ - M_t^- - (f(\vx_t^+) - f(\vx_t^-))
    \big|
\right]
&= \E\!\left[\Big|\frac{1}{b}\sum_{j=1}^b X_j\Big|\right]
\le \sqrt{\E\!\left[\Big(\frac{1}{b}\sum_{j=1}^b X_j\Big)^2\right]}
\le \frac{4 \eta G m}{\sqrt{b}}.
\label{eq:strong-assumptions-variance}
\end{align}
Substituting~\eqref{eq:strong-assumptions-variance} into
Theorem~\ref{thm:general-descent} directly yields the one-step bound
\eqref{eq:strong-assumptions-one-step}.  
Averaging over $t$ and optimizing $\eta$ gives
\eqref{eq:strong-assumptions-avg}--\eqref{eq:strong-assumptions-optimized}.
Finally, plugging these quantities into the total function-call count and
minimizing over $m$ establishes the complexity bound
\eqref{eq:strong-assumptions-final}.
\end{proof}

\section{Learning with Helper or Human Feedback: Convergence Proof}

We consider the setting where exact function evaluations are unavailable, and the
algorithm receives \emph{comparison feedback} from a helper (e.g., a human or a
proxy). The helper returns a scalar signal $h(\vx)$ such that pairwise
differences approximate true function differences.

\begin{assumption}[$\delta$-inexact comparison feedback]
\label{ass:delta-helper}
There exists $\delta\ge 0$ such that for all $\vx,\vy\in\R^d$,
\[
  \E\Big[\;\big|\,h(\vx)-h(\vy) \;-\; (f(\vx)-f(\vy))\,\big|\;\Big] \;\le\; \delta.
\]
\end{assumption}

At iteration $t$, define the perturbed points $\vx_t^\pm=\vx_t\pm \eta_t \vs_t$
and set the scalars
\[
  M_t^\pm \;:=\; h(\vx_t^\pm),
\]
so that the update rule of Algorithm~\ref{alg:srs} becomes
\(
  \vx_{t+1} = \vx_t - \eta_t\,\mathrm{sign}(M_t^+ - M_t^-)\,\vs_t.
\)
\begin{propbox}
\begin{theorem}[Convergence with $\delta$-inexact feedback]
\label{thm:helper-delta}
Suppose $f$ is $(L_0,L_1)$-smooth in the sense of \eqref{eq:L0L1-smooth},
the random directions satisfy Assumption~\ref{ass:distribution}
($\E\|\vs_t\|^2=1$ and $\E[|\langle \nabla f(\vx),\vs_t\rangle|]\ge
\mu_{\cD}\|\nabla f(\vx)\|$), and the helper satisfies
Assumption~\ref{ass:delta-helper}.
Let $F_0:=f(\vx_0)-f^\star<\infty$ and choose a constant stepsize
$\eta_t\equiv \eta \le \mu_{\cD}/L_1$. Then
\begin{equation}
\label{eq:helper-one-step}
\E[f(\vx_{t+1})]
\;\le\;
\E[f(\vx_t)]
\;-\;\frac{\mu_{\cD}}{2}\,\eta\,\E\|\nabla f(\vx_t)\|
\;+\;\frac{L_0}{2}\,\eta^2
\;+\; 2\,\delta.
\end{equation}
Averaging over $t=0,\dots,T-1$ yields
\begin{equation}
\label{eq:helper-avg-grad}
\frac{1}{T}\sum_{t=0}^{T-1}\E\|\nabla f(\vx_t)\|
\;\le\;
\frac{1}{\mu_{\cD}}
\left(
  \frac{F_0}{\eta T}
  + \frac{L_0}{2}\,\eta
  + \frac{2\delta}{\eta}
\right).
\end{equation}
\end{theorem}
\end{propbox}
Optimizing the right-hand side over $\eta$ gives
\begin{equation}
\label{eq:helper-optimized}
\frac{1}{T}\sum_{t=0}^{T-1}\E\|\nabla f(\vx_t)\|
\;=\;
\frac{1}{\mu_{\cD}}\,
\cO\!\left(
  \frac{L_1 F_0}{\mu_{\cD} T}
  + \sqrt{\frac{L_0 F_0}{T}}
  + \sqrt{L_0\,\delta}
\right).
\end{equation}
In particular, choosing
\(
T=\cO\!\big(\tfrac{dL_1}{\varepsilon}+\tfrac{dL_0 F_0}{\varepsilon^2}\big)
\)
(with $\mu_{\cD}\asymp d^{-1/2}$) ensures
\begin{equation}
\label{eq:helper-eps-floor}
\frac{1}{T}\sum_{t=0}^{T-1}\E\|\nabla f(\vx_t)\|
\;\le\;
\varepsilon \;+\; \cO(\sqrt{d\delta}).
\end{equation}
That is, convergence is guaranteed up to an $\cO(\sqrt{d\delta})$ accuracy floor
set by the inexactness of the helper.

\begin{proof}
Apply the general descent inequality
(Theorem~\ref{thm:general-descent}) with $M_t^\pm=h(\vx_t^\pm)$:
\begin{align*}
\E[f(\vx_{t+1})]
&\le
\E[f(\vx_t)]
- \tfrac{\mu_{\cD}}{2}\,\eta\,\E\|\nabla f(\vx_t)\|
+ \tfrac{L_0}{2}\,\eta^2
+ 2\,\E\!\Big[
  \big|\, (M_t^+ - M_t^-)-\big(f(\vx_t^+) - f(\vx_t^-)\big)\,\big|
\Big].
\end{align*}
By Assumption~\ref{ass:delta-helper} with $(\vx,\vy)=(\vx_t^+,\vx_t^-)$, the
last expectation is at most $2\delta$, giving
\eqref{eq:helper-one-step}. Summing from $t=0$ to $T-1$, taking total
expectation, and telescoping yields
\[
\mu_{\cD}\,\eta \sum_{t=0}^{T-1}\E\|\nabla f(\vx_t)\|
\;\le\;
F_0 \;+\; \tfrac{L_0}{2}\,\eta^2 T \;+\; 2\delta\,T,
\]
which rearranges to \eqref{eq:helper-avg-grad}. Optimizing the right-hand side
over $\eta$ under the constraint $\eta\le \mu_{\cD}/L_1$ (take
$\eta=\min\{\mu_{\cD}/L_1,\sqrt{F_0/(L_0T)},\sqrt{2\delta/L_0}\}$) gives
\eqref{eq:helper-optimized}. Finally, plugging
$T=\cO(\tfrac{dL_1}{\varepsilon}+\tfrac{dL_0 F_0}{\varepsilon^2})$ (using
$\mu_{\cD}\asymp d^{-1/2}$ for isotropic $\cD$) makes the first two terms
$\le \varepsilon$, and the remaining term is $\cO(\sqrt{\delta})$, yielding
\eqref{eq:helper-eps-floor}.
\end{proof}

\section{Failure of Classical Momenta}
\label{sec:failure-momentum}

We now examine whether classical momentum techniques from first-order
optimization can improve the performance of Random Search.
We consider three natural extensions of popular momentum formulations:
\emph{Heavy-Ball}~\cite{polyak1964some},
\emph{momentum-based variance reduction (MVR)}~\cite{cutkosky2019momentum},
and \emph{implicit gradient transport}~\cite{arnold2019reducing}.
Each method can be adapted to our comparison-based framework by
replacing stochastic gradients with function differences.

\paragraph{Momentum formulations.}
Let $\vs_t\sim\cD$ and define $\vx_t^\pm = \vx_t \pm \eta \vs_t$.
We consider the following recursions:
\begin{align*}
  \text{(Heavy-Ball)} & \quad
  M_t = (1-\beta)M_{t-1} + \beta\big(f_\xi(\vx_t^+) - f_\xi(\vx_t^-)\big), \\[4pt]
  \text{(MVR)} & \quad
  M_t = (1-\beta)\!\left(M_{t-1}
     + f_\xi(\vx_t^+) - f_\xi(\vx_t^-)
     + f_\xi(\vx_{t-1}^+) - f_\xi(\vx_{t-1}^-)\right)
     + \beta\big(f_\xi(\vx_t^+) - f_\xi(\vx_t^-)\big), \\[4pt]
  \text{(Transport)} & \quad
  M_t = (1-\beta)M_{t-1}
     + \beta\big(f_\xi(\tilde{\vx}_t^+) - f_\xi(\tilde{\vx}_t^-)\big),
     \quad
     \tilde{\vx}_t^\pm = \vx_t^\pm + \tfrac{1-\beta}{\beta}(\vx_t^\pm - \vx_{t-1}^\pm).
\end{align*}
Each variant attempts to smooth the stochastic signal
$f_\xi(\vx_t^+)-f_\xi(\vx_t^-)$ across iterations.

\paragraph{Error decomposition.}
For clarity, we focus on the Heavy-Ball recursion; the same reasoning applies to
the others. Define the momentum error
\[
e_t := M_t - \big(f(\vx_t^+) - f(\vx_t^-)\big),
\]
and decompose the updates as
\[
e_t = (1-\beta)\big(e_{t-1} + b_t\big) + \beta v_t,
\]
where
\[
b_t := f(\vx_t^+) - f(\vx_t^-) - \big(f(\vx_{t-1}^+) - f(\vx_{t-1}^-)\big),
\qquad
v_t := \big(f_\xi(\vx_t^+) - f_\xi(\vx_t^-)\big)
      - \big(f(\vx_t^+) - f(\vx_t^-)\big).
\]

\paragraph{Bias--variance tradeoff.}
Assuming $f$ is $G$-Lipschitz, we have $|b_t|\le 2G\eta$ and
$\E[v_t^2]\le (2G\eta)^2$.
Taking expectations and unrolling the recursion gives
\begin{align*}
    \E|e_t|
    &\le (1-\beta)^t\E|e_0|
      + (1-\beta)\!\sum_i (1-\beta)^{t-i}\E|b_i|
      + \beta\,\sqrt{\E\!\left[\left(\sum_i (1-\beta)^{t-i}v_i\right)^2\right]} \\[4pt]
    &\le (1-\beta)^t\E|e_0|
      + (1-\beta)\tfrac{2G\eta}{\beta}
      + 2G\eta\sqrt{\beta}.
\end{align*}
Both the bias and variance terms scale linearly in $\eta$, and no choice of
$\beta$ can improve this dependence. The optimal value is $\beta=1$, which
eliminates momentum altogether. The root cause is that the bias and variance
of the difference estimator are of the same order in $\eta$, preventing the
variance reduction mechanism from dominating as in first-order methods.

\paragraph{Other variants.}
When each $f_\xi$ is $G$-Lipschitz, the same argument applies to the MVR
recursion: both the correction and the variance term scale as
$\cO(G\eta)$, leaving no improvement over the non-momentum baseline.

The implicit transport momentum fails for a different reason.
It requires extrapolated states
\(
\tilde{\vx}_t^\pm = \vx_t^\pm + \tfrac{1-\beta}{\beta}(\vx_t^\pm - \vx_{t-1}^\pm),
\)
which introduces an additional displacement of order
$\tfrac{1-\beta}{\beta}\eta$. Since the variance of the difference
$f_\xi(\tilde{\vx}_t^+) - f_\xi(\tilde{\vx}_t^-)$ depends on the proximity
of these states, this extrapolation \emph{increases} variance by a factor of
$\tfrac{1}{\beta}$. Consequently, the overall noise term scales as
$\cO(\tfrac{\sigma}{\sqrt{\beta}})$, offsetting the potential
$\sqrt{\beta}$ gain from averaging.  
Once again, the best choice is $\beta=1$, implying that not using momentum
is preferable.

\paragraph{Conclusion.}
In summary, all classical momenta---Heavy-Ball, momentum-based variance
reduction, and implicit transport---fail to improve Random Search.  
The key difficulty lies in the structure of function-difference oracles:
both the bias and the variance scale with the step size $\eta$, unlike in
gradient-based methods where variance dominates.  
Consequently, naive momentum adaptation offers no benefit and can even increase
noise.

\end{document}

%% file: math_commands.tex
\usepackage{amsmath,amsfonts,bm}

\def\eqref#1{equation~\ref{#1}}

\def\1{\bm{1}}

\def\cO{{\mathcal{O}}}
\def\cD{{\mathcal{D}}}
\def\cL{{\mathcal{L}}}

\def\vzero{{\bm{0}}}

\def\vg{{\bm{g}}}

\def\vs{{\bm{s}}}

\def\vx{{\bm{x}}}
\def\vy{{\bm{y}}}
\def\vz{{\bm{z}}}

\def\O{{\mathcal{O}}}

\DeclareMathAlphabet{\mathsfit}{\encodingdefault}{\sfdefault}{m}{sl}
\SetMathAlphabet{\mathsfit}{bold}{\encodingdefault}{\sfdefault}{bx}{n}

\newcommand{\E}{\mathbb{E}}

\newcommand{\R}{\mathbb{R}}

\newcommand{\Var}{\mathrm{Var}}

\newcommand{\Cov}{\mathrm{Cov}}
